\documentclass[12pt]{article}

\usepackage{graphicx}
\usepackage[a4paper, hmargin={2.5cm,2.5cm},vmargin={3.2cm,3.2cm}]{geometry}
\usepackage[font=small,labelfont=bf]{caption}
\usepackage{subcaption}

\usepackage[T1]{fontenc}
\usepackage[utf8]{inputenc}
\usepackage[english]{babel}

\usepackage{amsmath,amssymb,amsfonts}
\usepackage{amsthm,latexsym}
\usepackage{etoolbox}
\usepackage{extarrows}

\usepackage{cite}
\usepackage{paralist}
\usepackage{tikz}
\usetikzlibrary{shapes}
\usepackage{pgfplots}
\pgfplotsset{compat=newest}
\usepackage{twoopt}

\usepackage{url}
\usepackage[bf,compact,pagestyles,small]{titlesec}
\titlespacing*{\section}{0pt}{14pt}{4pt}
\titlespacing*{\subsection}{0pt}{8pt}{3pt}

\makeatletter
\patchcmd{\ttlh@hang}{\parindent\z@}{\parindent\z@\leavevmode}{}{}
\patchcmd{\ttlh@hang}{\noindent}{}{}{}
\makeatother

\usepackage{fancyhdr,lastpage}
\def\maketimestamp{\count255=\time
\divide\count255 by 60\relax
\edef\thetime{\the\count255:}%
\multiply\count255 by-60\relax
\advance\count255 by\time
\edef\thetime{\thetime\ifnum\count255<10 0\fi\the\count255}
\edef\thedate{\number\day-\ifcase\month\or Jan\or Feb\or Mar\or
             Apr\or May\or Jun\or Jul\or Aug\or Sep\or Oct\or
             Nov\or Dec\fi-\number\year}
\def\timstamp{\hbox to\hsize{\tt\hfil\thedate\hfil\thetime\hfil}}}
\maketimestamp
\fancypagestyle{plain}{%
\fancyhf{} 
\fancyfoot[C]{\small \thepage{} of \pageref{LastPage}}}

\numberwithin{equation}{section}  
\allowdisplaybreaks[4]
\newtheorem{theorem}{Theorem}[section]
\newtheorem{lemma}[theorem]{Lemma}
\newtheorem{proposition}[theorem]{Proposition}
\newtheorem{corollary}[theorem]{Corollary}
\theoremstyle{definition}
\newtheorem{definition}[theorem]{Definition} 

\newtheorem{example}{Example}
\theoremstyle{remark}
\newtheorem{remark}{Remark}

\DeclareMathOperator{\Span}{span} %

\DeclareMathOperator{\covol}{covol}
\newcommand{\mean}[1]{\ensuremath{M( {#1} )}}

\newcommand{\gsi}[1][g]{\ensuremath{\seq{T_\gamma {#1}_j}_{j \in J,
      \gamma \in \Gamma_j}}}

\newcommandtwoopt{\gaborG}[3][a][b]{\mathcal{G}(#3,#1,#2)} 

\newcommand{\card}[1]{\# \abs{#1}} 

\newcommand{\eps}{\ensuremath{\varepsilon}}

\newcommand*{\numbersys}[1]{\ensuremath{\mathbb{#1}}}
\newcommand*{\C}{\numbersys{C}}
\newcommand*{\R}{\numbersys{R}}

\newcommand*{\Q}{\numbersys{Q}}

\newcommand*{\Z}{\numbersys{Z}}

\newcommand*{\N}{\numbersys{N}}

\newcommand{\itvcc}[2]{\ensuremath{\left[{#1},{#2}\right]}} %
\newcommand{\abs}[1]{\ensuremath{\left\lvert#1\right\rvert}}
\newcommand{\abssmall}[1]{\ensuremath{\lvert#1\rvert}}
\newcommand{\absbig}[1]{\ensuremath{\bigl\lvert#1\bigr\rvert}}
\newcommand{\absBig}[1]{\ensuremath{\Bigl\lvert#1\Bigr\rvert}}
\newcommand{\norm}[2][]{\ensuremath{\left\lVert#2\right\rVert_{#1}}}

\newcommand{\innerprod}[3][]{\ensuremath{\left\langle #2,#3\right\rangle_{\! #1}}}

\newcommand{\innerprodbig}[3][]{\ensuremath{\bigl \langle #2,#3\bigr\rangle_{\!\!#1}}}
\newcommand{\innerprods}[2]{\ensuremath{\langle #1,#2\rangle}}

\newcommand{\seq}[1]{\ensuremath{\left(#1\right)}}
\newcommand{\setprop}[2]{\ensuremath{\left\lbrace{#1} : {#2}\right\rbrace}}

\newcommand\cD{\mathcal{D}_E}
\newcommand\AP{\mathcal{A}}

\newcommand{\lat}[1]{\ensuremath {#1}} 
\newcommand{\LG}{\ensuremath\lat{\Gamma}}

\newcommand{\ghat}{\widehat{G}}
 
\makeatletter
\def\moverlay{\mathpalette\mov@rlay}
\def\mov@rlay#1#2{\leavevmode\vtop{%
   \baselineskip\z@skip \lineskiplimit-\maxdimen
   \ialign{\hfil$\m@th#1##$\hfil\cr#2\crcr}}}
\newcommand{\charfusion}[3][\mathord]{
    #1{\ifx#1\mathop\vphantom{#2}\fi
        \mathpalette\mov@rlay{#2\cr#3}
      }
    \ifx#1\mathop\expandafter\displaylimits\fi}
\makeatother

\newcommand{\bigcupdot}{\charfusion[\mathop]{\bigcup}{\cdot}}

\usepackage{hyperref}
\hypersetup{
 pdfview={FitH},
 pdfstartview={FitH},
 pdfauthor = {Jakob Lemvig},
 pdftitle = {},
 pdfkeywords = {},
 pdfcreator = {LaTeX with hyperref package},
 pdfproducer = PDFlatex}

\makeatletter
\def\blfootnote{\xdef\@thefnmark{}\@footnotetext}
\def\subjclass{\xdef\@thefnmark{}\@footnotetext}
\long\def\symbolfootnote[#1]#2{\begingroup%
\def\thefootnote{\fnsymbol{footnote}}\footnote[#1]{#2}\endgroup}
\if@titlepage
  \renewenvironment{abstract}{%
      \titlepage
      \null\vfil
      \@beginparpenalty\@lowpenalty
      \begin{center}%
        \bfseries \abstractname
        \@endparpenalty\@M
      \end{center}}%
     {\par\vfil\null\endtitlepage}
\else
  \renewenvironment{abstract}{%
      \if@twocolumn
        \section*{\abstractname}%
      \else
        \small
        \list{}{%
          \settowidth{\labelwidth}{\textbf{\abstractname:}}
          \setlength{\leftmargin}{50pt}
          \setlength{\rightmargin}{50pt}
          \setlength{\itemindent}{\labelwidth}
          \addtolength{\itemindent}{\labelsep}
        }
        \item[\textbf{\abstractname:}]

      \fi}
      {\if@twocolumn\else\endlist\fi}
\fi
\makeatother

\begin{document}

\title{System bandwidth and the existence of generalized shift-invariant frames}

\date{\today}

 \author{Hartmut F\"uhr\footnote{Lehrstuhl A f\"ur Mathematik, RWTH
     Aachen, E-mail:
     \protect\url{fuehr@matha.rwth-aachen.de}}\phantom{$^\ast$} and  Jakob Lemvig\footnote{Technical University of Denmark, Department of Applied Mathematics and Computer Science, Matematiktorvet 303B, 2800 Kgs.\ Lyngby, Denmark, E-mail: \protect\url{jakle@dtu.dk}}} 


\maketitle

\thispagestyle{plain}

 \begin{abstract} 
 We consider the question whether, given a countable system of lattices
 $(\Gamma_j)_{j \in J}$ in a locally compact abelian group $G$, there
 exists a sequence of functions $(g_j)_{j \in J}$ such that the
 resulting generalized shift-invariant system $(g_j(\cdot - \gamma))_{j \in
   J, \gamma \in \Gamma_j}$ is a tight frame of $L^2(G)$. This paper develops a
 new approach to the study of almost periodic functions for
 generalized shift-invariant systems based on an \emph{unconditionally
 convergence property}, replacing previously used local integrability
 conditions. From this theory, we derive characterizing
 relations for tight and dual frame generators, we introduce the \emph{system
   bandwidth} as a measure of the total bandwidth a generalized shift-invariant system can carry, and we show that the
 so-called Calder\'on sum is uniformly bounded from below for
 generalized shift-invariant frames. We exhibit a condition on the
 lattice system for which the unconditionally
 convergence property is guaranteed to hold. 
Without the unconditionally
 convergence property, we show, counter intuitively, that even orthonormal bases
 can have arbitrary small system bandwidth.  
Our results show that the question of existence of frame generators
for a general lattice system can be rather subtle, depending on analytical properties, such as the system bandwidth, as well as on algebraic properties  of the lattice system.  
 \end{abstract}


\section{Statement of the problem}

\subsection{Some terminology}

Let us start by recalling some facts from harmonic analysis on locally compact abelian (LCA) groups; for a thorough introduction, we refer to \cite{Rudin_Fourier,HeRo}.
Throughout the paper, we let $G$ denote a second countable LCA group. It is endowed with a translation-invariant Radon measure, unique up to normalization, called the Haar measure of $G$, and denoted by $\mu_G$. We let $L^2(G)$ denote the Hilbert space of square-integrable functions with respect to Haar measure, and $C_b(G)$ the space of bounded continuous functions.  We will typically write LCA groups additively, and let $0 \in G$ denote the neutral element. 

A \emph{lattice}, sometimes called a uniform lattice, in $G$ is a discrete subgroup $\Gamma < G$ with the property that the quotient $G/\Gamma$ is  compact.  Since we assume $G$ to be second countable, lattices in $G$ are necessarily countable. 
A \emph{generalized shift-invariant} (GSI) system in $L^2(G)$ is constructed by picking a family of lattices $\Gamma_j \subset G$ and a family of vectors $(g_j)_{j \in J} \subset L^2(G)$, and defining the family 
\[
  (T_\gamma g_j)_{j \in J, \gamma \in \Gamma_j},
\]
where $T_\gamma f = f(\cdot -\gamma)$ denotes the translation operator
on $L^2(G)$. 
This general class of systems of vectors was introduced in
\cite{MR1916862,MR2132766} for $G =\R^n$, and further studied, e.g., in
\cite{JakobsenReproducing2014,MR2283810} for the general setting of
LCA groups.

 GSI systems can be seen as
countable filter banks or adaptive time-frequency representations. 
They are interesting objects in their own right and not only as
a framework to unify Gabor and wavelet analysis. We refer to
\cite{MR2854065} for an implementation and applications of GSI systems in
signal processing, to \cite{MR3627474} for a construction of dual GSI frames
for $L^2(\R)$, and to \cite{1704.07176,1606.08647} for sparseness
properties of GSI frames
for $L^2(\R)$. 

Next, some terminology relating to frames, Bessel systems, and related notions. A family of vectors $(\eta_i)_{i \in I}$ contained in a Hilbert space $\mathcal{H}$ is called a \emph{Bessel system} if there exists a constant $B$ such that, for all $f \in \mathcal{H}$
\[
 \sum_{i \in I} \left| \langle f, \eta_i \rangle \right|^2 \le B \| f \|^2. 
\] 
The constant $B$ is called a \emph{Bessel bound} of the system.
If, in addition, there exists a lower bound $A>0$ such that, for all $f \in \mathcal{H}$,
\begin{equation}
 A \| f \|^2 \le \sum_{i \in I} \left| \langle f, \eta_i \rangle
 \right|^2 ,\label{eq:abstract-frame-lower}
\end{equation}
the system is called a \emph{frame}. 
The constants $A$ and $B$ are called \emph{frame bounds}; the optimal
frame bounds, denoted $A^\dagger$ and $B^\dagger$, respectively, are the largest possible value for $A$ and the smallest
possible value for $B$ in the above inequalities. If $A=B=1$,
the frame is said to be tight. If $A=B=1$, the frame is called a \emph{Parseval frame}. 
As a particular case of frames, we mention orthonormal bases, which
can be characterized as Parseval frames with normalized
elements. 

For a Bessel system $(\eta_i)_{i \in I}$ the frame operator
 on $\mathcal{H}$ is given as $S_\eta=\sum_i
 \innerprod{\cdot}{\eta_i}\eta_i$; this operator is bounded and,
 furthermore, invertible if the lower bound
 \eqref{eq:abstract-frame-lower} holds. Two Bessel systems
 $(\eta_i)_{i \in I}$ and $(\kappa_i)_{i \in I}$ are called \emph{dual
   frames} if 
 \begin{equation}
 f = \sum_{i \in I} \innerprod{f}{\eta_i}\kappa_i \quad \text{for all
 } f \in \mathcal{H};\label{eq:dual-frames-def}
\end{equation}
in this case the two Bessel systems are automatically frames. 
Finally,
a frame $(\eta_i)_{i \in I}$ has at least one dual frame
$(\kappa_i)_{i \in I}$ so that \eqref{eq:dual-frames-def} holds; the canonical choice is
$\kappa_i=S_\eta^{-1}\eta_i$ for all $i \in I$. 

A \emph{generalized shift-invariant frame} is a generalized
shift-invariant system that is a frame at the same time. From the
general frame theory outlined above given a GSI frame $(T_\gamma
g_j)_{j,\gamma}$, there exists a dual frame
$(\widetilde{g}_{j,\gamma})_{j,\gamma} \subset L^2(G)$, i.e,  for all $f \in L^2(G)$, we have
\[
f = \sum_{j\in J}\sum_{\gamma \in \Gamma_j} \langle f, T_\gamma g_j \rangle
  \widetilde{g}_{j,\gamma} = \sum_{j\in J}\sum_{\gamma \in \Gamma_j} \langle f, \widetilde{g}_{j,\gamma} \rangle T_\gamma g_j~,
\]
with unconditional convergence. Parseval frames are
characterized by the property that one may take
$\widetilde{g}_{j,\gamma} = T_\gamma g_j$. However, for a non-tight
GSI frame there might not exist any dual frames with GSI
structure. Recall that a Riesz basis is
  a non-redundant frame and that Riesz bases only have one dual frame,
  namely, the canonical dual.
\begin{proposition}
\label{prp:GSI-frames-no-GSI-dual}
    The dual basis of a GSI Riesz basis need not have GSI structure. 
\end{proposition}
\begin{proof}
  We consider dyadic wavelet systems $\seq{\eta_{j,k}}_{j,k\in \Z}$
  in  $L^2(\R)$, where $\eta_{j,k}=D_{2^j} T_k \eta$ and $D_a
  f=a^{1/2}f(a\cdot)$ for $a>0$ and $f \in L^2(\R)$. We will let $\psi$ be a
  generator of an orthonormal wavelet basis $\seq{\psi_{j,k}}_{j,k\in
    \Z}$. Furthermore, we assume that $\psi$ is a continuous, compactly supported
  function with a unique maximum at $x=0$. Daubechies~\cite[p.~989]{MR1066587} and Chui and Shi~\cite[Section
  3]{MR1199539} prove that the canonical dual of the dyadic wavelet Riesz
  basis generated by $\eta=\psi+\epsilon D_2 \psi$ for
  $0<\epsilon <1$, where $\psi$ is any orthonormal wavelet, does not have
  wavelet structure. Even more is true: it does not have GSI
  structure.   In fact,
  the canonical dual basis $\seq{S^{-1}_\eta \eta_{j,k}}_{j,k\in \Z}$
    of $\seq{\eta_{j,k}}_{j,k\in \Z}$ can be computed explicitly as
 \begin{equation}
    \label{eq:app-riesz-basis-dual}
   S_\eta^{-1} \eta_{j,k} 
= \begin{cases}
     \psi_{j,k} - \eps \psi_{j-1,k/2} + \dots + (-\eps)^n
     \psi_{j-n,k/2^n} & j \in \Z, k \in \Z \setminus \{0\}, \\
       \sum_{m=0}^\infty (-\eps)^m \psi_{j-m,0} & j \in \Z, k=0,
    \end{cases}
  \end{equation}
where $n= \sup \{ m \in\N_0: 2^m \vert k\}$.
  Now, it can be seen by using the properties of
  $\seq{\psi_{j,k}}_{j,k \in \Z}$
  that there exist dual basis vectors that are not translates of any other
  dual basis vector. Indeed, $S_\eta^{-1} \eta_{j,k}$, $k \neq 0$, is
  compactly supported. On the other hand, $\sum_{m=0}^\infty (-\eps)^m
  \psi_{j-m,0}$, $j \in \Z$, has a unique maximum at $x=0$, but these functions are
  not compactly supported.  Every GSI system containing these functions 
  must contain elements that are not compactly supported, but have a unique maximum
  at a point different from zero, whereas the dual basis contains no such function. 
\end{proof}

Due to Proposition~\ref{prp:GSI-frames-no-GSI-dual}, we introduce the notion of a \emph{dual
  generalized shift-invariant system} consisting of a system
$\mathcal{G}$ of lattices, and two families $(g_j)_{j \in J}, (h_j)_{j
  \in J} \subset L^2(G)$ such that the generalized shift-invariant
systems fulfill the following: $(T_\gamma g_j)$ and $(T_\gamma h_j)$ are Bessel sequences, and
\[
 f = \sum_{j,\gamma} \langle f, T_\gamma g_j \rangle T_\gamma h_j
\] holds for all $f \in L^2(G)$.

\subsection{Aims of this paper}

The starting point of this paper is a system of lattices $\mathcal{G} = (\Gamma_j)_{j \in J}$ in $G$. 
We want to find sufficient and/or necessary criteria on $\mathcal{G}$
for the existence of an associated system of vectors $(g_j)_{j \in J}$
 such that the generalized shift-invariant system arising from these
 data is a (tight) frame. We then call the system $(g_j)_{j \in J}$
 \emph{(tight) frame generators} for $\mathcal{G}$. In the case of
 existence of 
 dual frames, we call the systems $(g_j)_{j \in J}$
and $(h_j)_{j \in J}$ \emph{dual frame generators}. The associated lower and upper frame bounds shall be denoted by $A_g,B_g$, etc. 

Stated in such general terms, the problem of deciding the existence of frame generators seems somewhat impenetrable at first. We will not be able to fully solve the existence problem for generating systems, but we will derive results and construct examples showing that this question is remarkably subtle, involving both analytic and algebraic aspects.

An analytic condition that we shall investigate has to do with
bandwidth. To motivate this notion, it is useful to recall the Shannon
Sampling Theorem. Pick an interval $I = [\xi,\xi+L] \in \R$, and
define the closed subspace $PW_I = \{ f \in L^2(\R) : \hat{f} \cdot
\mathbf{1}_I  = \hat{f} \} \subset L^2(\R)$, where $\mathbf{1}_I$
denotes the characteristic function on $I$. Then, letting $g =
L^{-1/2} ({\mathbf{1}}_I)^\vee$, where $(\cdot)^\vee$ denotes the inverse
Fourier transform, we find that the system $(T_{k/L}g)_{k \in \Z}$ is an orthonormal basis of $PW_I$. The length $L$ is commonly called the \emph{bandwidth} of the space $PW_I$. 
We now revert this view. Starting from a lattice $\Gamma = c \Z$, we
pick an interval $I$ of length $L = 1/c$, and a generator  $g \in
PW_I$ such that $(T_{ck} g)_{k \in \Z}$ is an orthonormal basis of the
Paley-Wiener space $PW_I$. Now, given a \emph{system} of lattices
$(c_j \Z)_{j \in J}$, one possible strategy for the construction of
compatible tight frame (in fact, orthonormal basis) generators
$(g_j)_{j \in J}$ would be to cover the real line (the frequency
domain) disjointly by intervals of length $L_j = 1/c_j$, and pick orthonormal basis generators for each $PW_{I_j}$. The question remains, however, whether the combined intervals suffice to cover the full real axis, i.e., whether the \emph{system bandwidth}, defined by $\sum_{j \in J} L_j$, is infinite. These considerations motivate the following definition.
\begin{definition}
 Let $\mathcal{G} = (\Gamma_j)_{j \in J}$ be a system of lattices in
 $G$. Then the quantity, where $\covol(\Gamma_j)$ is the Haar measure of a fundamental domain
of $\Gamma_j$ in $G$,
 \[
  BW(\mathcal{G}) = \sum_{j \in J} \frac{1}{\covol(\Gamma_j)} \in (0,\infty]
 \] is called the \emph{bandwidth} of $\mathcal{G}$,
\end{definition}

Now the above discussion suggests to study the relationship of the
existence of tight frame generators to the condition that $BW(\mathcal{G}) \ge
\mu_{\ghat} (\ghat)$. We shall exhibit situations in which
the bandwidth criterion is quite sharp, and other somewhat
pathological cases, in which bandwidth is an irrelevant quantity. Hence a general characterization of lattice system admitting dual frame generators will have to involve both analytical, quantitative criteria (such as bandwidth), as well as algebraic ones. 

The main results of this paper can be summarized as follows. We first introduce the approach to the analysis of GSI systems via almost periodic functions, established for wavelet systems by Laugesen in
works~\cite{MR1866351,MR1940326}, and then further generalized to the GSI
setting in~\cite{MR1916862}. We add a new result to this general
approach that allows to derive characterizing relations for dual
frame generators under suitable, rather mild, unconditional convergence conditions
(Theorem~\ref{thm:meta_convergence_AP}), and show, in Example~\ref{ex:UCP_LIC}, that it properly
generalizes the known characterizing results for GSI frames~\cite{JakobsenReproducing2014,MR2283810,MR1916862}. Our mild 
convergence conditions replace previously used local integrability
conditions. Under this convergence property we prove that the so-called
Calder\'o{}n sum for GSI frames is bounded from below by the lower frame bound
(Theorem~\ref{thm:a-lic-calderon-bounds}) which provides a necessary
condition on $(g_j)$ for the frame property, but which is also of
independent interest. It is also under the unconditional convergence property, we
prove that $BW(\mathcal{G}) \ge
\mu_{\ghat} (\ghat)$  is necessary for the existence of tight frame
generators (Theorem~\ref{thm:density-GTI}). We then present a general existence result for
frames assuming the existence of a suitable dual covering
(Theorem~\ref{thm:shannon}). 

In absence of the unconditional convergence
property we construct tight GSI frames with
arbitrarily small bandwidth (Theorem~\ref{thm:ONB_fbw}). Using the
notion of independent lattices, we then exhibit a rather general class
of lattice families for which the unconditional convergence property
has an easy characterization and for which the characterizing relations from Theorem~\ref{thm:meta_convergence_AP} are rather restrictive (Theorem~\ref{thm:rat_ind}). Further illustrations of various interesting features of the problem studied in this paper can be found in Examples~\ref{exa:reindex},~\ref{exa:robust} and~\ref{exa:increase}. 

\begin{remark}
 \label{rem:countable_families}
 In our considerations, taking some $g_j$ to be the zero function is
 expressly allowed, unless we want to construct orthonormal
 bases. Hence, whenever the existence of Bessel, frame, or dual frame generators is shown for a subfamily of a family $\mathcal{G}$ of lattices, it holds for $\mathcal{G}$ itself.  Thus one should be aware that the following sufficient conditions only need to be fulfilled by a suitable subfamily of the original lattice family. 
The LCA group $G$ being second countable implies that $L^2(G)$ is
separable, and thus all (discrete) frames in this space are countable. Hence, we will concentrate on countable lattice families $\Gamma$. 
\end{remark}


\section{Notation for LCA groups}
\label{sec:notation-lca-groups}

As stated above, $G$ will always denote a second countable, locally
compact abelian group, its Haar measure will be denoted by $\mu_G$. 

A \emph{fundamental domain}, also known as a Borel section, associated to a lattice $\Gamma \subset G$ is a Borel set $K \subset G$ such that the $\Gamma$-translates tile $G$ up to sets of measure zero; such sets always exist. A more rigorous formulation of this is as follows: Let $\mathbf{1}_K$ denote the indicator function of $K$. Then $K$ is a fundamental domain for $\Gamma$ if
\[
  \sum_{\gamma  \in \Gamma} \mathbf{1}_{K}(x+\gamma) = 1 \quad (\text{a.e. } x \in G)~.
\] 
It is an easy exercise, using translation invariance of  Haar measure,
to prove that for any two fundamental domains $K,K'$ of the same
lattice $\Gamma$, one has $\mu_G(K) = \mu_G(K')$. The \emph{covolume}
$\covol{(\Gamma)}$ of $\Gamma$ in $G$ is then defined as
$\mu_G(K)$. Fundamental domains can always be chosen to be pre-compact. 

For $G = \mathbb{R}^n$, all lattices are given by $\Gamma = C \Z^n$, where $C$ can be any invertible matrix.  
Since the cube $[0,1)^n$ is a fundamental domain for $\Z^n$, it is immediate that $C [0,1)^n$ is a fundamental domain for $\Gamma = C \Z^n$, and one obtains $\covol(\Gamma) = |\det(C)|$. 

We let $\ghat$ denote the character group of $G$, i.e., the
group of all continuous homomorphisms $G \to \mathbb{T}$. The duality
between $G$ and $\ghat$ is denoted by $\langle \cdot, \cdot
\rangle : G \times \ghat \to \mathbb{T}$; confusion of this notation with inner products in $L^2$ will be cleared up by the context.  
The Fourier transform of a function $f \in L^1(G)$ is then given by $\hat{f} : \ghat \to \mathbb{C}$,
\[
 \hat{f}(\omega) = \int_{G} f(x) \overline{\langle x, \omega \rangle} dx~.
\] This defines a bounded operator $\mathcal{F} : L^1(G) \to C_b(G)$,
$f \mapsto \hat{f}$. The Plancherel theorem states that, after proper
normalization of the Haar measure on $\ghat$, the operator
$\mathcal{F}|_{L^1(G) \cap L^2(G)}$ extends uniquely to a unitary
operator from $L^2(G)$ onto $L^2(\ghat)$ which we will also denote by $\mathcal{F}f=\hat{f}$. 

Given a lattice $\Gamma $, its dual lattice (or annihilator) is given by 
\[
\Gamma^\bot = \{ \alpha \in \ghat: \langle \alpha, \gamma \rangle = 1
\quad \forall \gamma \in \Gamma \}~.
\] 
By duality theory, $\Gamma^\bot < \ghat$ is a lattice as
well. In fact, if one normalizes the Haar measure on $\ghat$ in
such a way that the Plancherel theorem holds, then the covolumes of $\Gamma$ and $\Gamma^\bot$ are related by
\[
 \covol(\Gamma) \covol(\Gamma^\bot) = 1~. 
\]
In the case $G =\R^n$ and $\Gamma = C \Z^n$ for some invertible matrix $C$, the dual lattice is computed as $\Gamma^\bot = C^{-T} \Z^n$, with $C^{-T}$ denoting the inverse transpose of $C$. 

To summarize, we let a  Haar measure on $G$ be given. We assume dual
measures so that Plancherel theorem holds, and we assume the
counting measure on discrete subgroups $\Gamma$ and choose the Haar
measure on $G / \Gamma$ as the quotient measure so that Weil's
integral formula holds. Using this quotient measure $\mu_{G/ \Gamma}$
on $G / \Gamma$, we can express the covolume as $\covol(\Gamma) = \mu_{G / \Gamma} (G / \Gamma)$. The quantity
$1/\covol (\Gamma)$ is sometimes called the density of the
subgroup, while $\covol(\Gamma)$ is called the lattice size.

\section{Almost periodic functions and GSI systems in $L^2(G)$}
\label{sec:almost-peri-funct}

\subsection{Fourier analysis of GSI systems}
\label{sec:fourier-analysis-gsi}

In order to understand the role of almost periodic functions, let us
fix a dual GSI system given by the lattice system $(\Gamma_j)_{j \in
  J}$ and the associated functions $(g_j)_{j \in J}$ and $(h_j)_{j \in
  J}$. We fix a closed set $E \subset \widehat{G}$ of measure zero, and define 
\begin{equation}
  \label{eq:def-D}
  \cD = \big\{ f \in L^2 (G) \, : \, \hat{f} \in L^\infty(\ghat) \text{ and }
    \exists K \subset \ghat\setminus E \text{ compact with } \hat{f} \mathbf{1}_K = \hat{f} \text{ a.e.}  \big\}.
\end{equation}
This is a translation-invariant and dense subspace of $L^2(G)$, and since the frame operator is
bounded precisely when the associated system is a Bessel system, the
Bessel property and
further frame theoretical
properties of the system 
only need to be checked on $\cD$. Here $E \subset \hat{G}$ denotes the \emph{blind spot of the system} \cite{JakobsenReproducing2014};  the specific choice of $E$ depends on
the application.

For $f \in \cD$, we define the functions 
$w_{f;g,h,j}:G\to \C$ for $j \in J$ by
\begin{equation}
 w_{f;g,h,j}(x) =  \sum_{\gamma \in \Gamma_j} \langle T_x f,
 T_{\gamma} g_j \rangle \langle T_{\gamma} h_j , T_x f \rangle . \label{eq:wfj-def}
\end{equation}
For each $j \in J$, the series in \eqref{eq:wfj-def} converge pointwise to a continuous
limit function as is seen by the following result. The result is a
dual version of  \cite[Lemma 3.4]{MR2283810} which is an adaptation of \cite[Lemma 2.2]{MR1916862}. 
\begin{lemma} \label{lem:wfj-convergence}
Fix $f \in \cD$ and $j \in J$. Let $\Gamma_j$ be a lattice in $G$ and $g_j,h_j \in L^2(G)$. Then $w_{f;g,h,j}$ is a trigonometric polynomial. More precisely,
  \[
   w_{f;g,h,j} (x) = \sum_{\alpha \in \Gamma_j^\bot} d_{j,\alpha} \langle \alpha, x \rangle ~,
  \]
  where 
  \[
   d_{j,\alpha} = \frac{1}{\covol(\Gamma_j)}\int_{\ghat} \hat{f}(\omega) \overline{\hat{g}_j(\omega) \hat{f}(\omega+\alpha)} \hat{h}_j(\omega+\alpha) d\omega~.
  \]
  In particular, $d_{j,\alpha} = 0$  for all but finitely many $\alpha \in \Gamma_j^\bot$. 
 \end{lemma}

We define the function $w_{f;g,h}$ on $G$ as 
\begin{equation} \label{eqn:defw2}
 w_{f;g,h}(x) = \sum_{j \in J} w_{f;g,h,j}(x) = \sum_{j \in J} \sum_{\gamma \in \Gamma_j} \langle T_x f, T_{\gamma} g_j \rangle \langle T_{\gamma} h_j , T_x f \rangle 
\end{equation}
provided that the series converge.
In the case $h_j = g_j$, for all $j \in J$, we write $w_{f;g} =
w_{f;g,h}$ 
and $w_{f;g,j} = w_{f;g,h,j}$. Without any
further assumption on the lattice system $\mathcal{G}$ and the
generators, we can only say that the series $\sum_{j \in
  J} w_{f;g,j}(x)$ converges in $\itvcc{0}{\infty}$, hence $ w_{f;g}:G
\to \itvcc{0}{\infty}$ is well-defined, while $w_{f;g,h}:G\to \C$
might not be.

However, under the Bessel property,
the partial sums $\sum_{j \in
  J} w_{f;g,h,j}$ converge pointwise absolutely and uniformly on
compact sets, and
the limit functions are continuous and bounded. 
To prove these properties, we need the following well-known result.
\begin{lemma} \label{lem:calderon_upper}
 If $(\Gamma_j)$, $(g_j)_{j \in J}$ generates a Bessel system in $L^2(G)$, then
 \[
  \sum_{j \in J} \frac{1}{\covol(\Gamma_j)} |\hat{g}_j(\omega)|^2 \le
  B_g  \quad (\text{a.e. } \omega).
 \]
\end{lemma}
The sum  $\sum_{j \in J} \frac{1}{\covol(\Gamma_j)}
|\hat{g}_j(\omega)|^2$ is called the \emph{Calder\'on sum} of the GSI system in
accordance with wavelet analysis. For a proof of
Lemma~\ref{lem:calderon_upper} in $L^2(\R^n)$ we refer to \cite[Proposition
4.1]{MR1916862}, and for a proof in $L^2(G)$ we refer to \cite{JakobsenReproducing2014,MR2283810}. 
\begin{lemma}
 \label{lem:wf-convergence}
Fix $f \in \cD$. Let $(\Gamma_j)_{j,\in J}$ denote a system of
lattices, and $(g_j)_{j \in J}$, $(h_j)_{j \in J}$
  associated function systems of Bessel generators. Then $w_{f;g,h}$ is uniformly continuous. Moreover, the right-hand side of  
  \[
   w_{f;g,h} = \sum_{j \in J} w_{f;g,h,j}
  \]
 converges uniformly and unconditionally on compact sets. 
\end{lemma}

\begin{proof}
To begin with, note that the Bessel assumption on the generators guarantees that the sum defining $w_{f;g,h}$ converges pointwise absolutely. 
 We next prove uniform continuity of the limit. Given $x_1,x_2 \in G$, we compute
\begin{eqnarray*}
 \lefteqn{\left| w_{f;g,h}(x_1)-w_{f;g,h}(x_2) \right|} \\ & \le & \sum_{j \in J} \sum_{\gamma \in \Gamma_j} \left| \langle T_{x_1} f , T_\gamma g_j \rangle \langle T_\gamma h_j, T_{x_1} f \rangle - 
 \langle T_{x_2} f , T_\gamma g_j \rangle \langle T_\gamma h_j, T_{x_2} f \rangle \right| \\
 & \le & \sum_{j \in J} \sum_{\gamma \in \Gamma_j} |\langle (T_{x_1}-T_{x_2}) f ,  T_\gamma g_j \rangle \langle T_\gamma h_j, T_{x_1} f \rangle| +
 |\langle (T_{x_1} f ,  T_\gamma g_j \rangle \langle T_\gamma h_j, (T_{x_2} - T_{x_1}) f \rangle| \\
 & \le & B_g  \| (T_{x_1}- T_{x_2})  f  \|  \cdot B_h \| T_{x_1} f \|  +   B_g \| T_{x_1} f \| \cdot B_h \| (T_{x_1} - T_{x_2}) f \| \\
 & \le & 2 B_g B_h \| T_{x_1-x_2} f - f \|^2~,
 \end{eqnarray*}
 using the Bessel constants $B_g,B_h$ 
and that the regular representation $x \mapsto T_x$ is a
homomorphism. Since this representation is strongly continuous, for any
$\epsilon >0$ there exists a neighborhood $U$ of the identity element such that $\| T_{x_1-x_2} f - f \|^2 < \epsilon$ whenever $x_1-x_2 \in U$. This shows uniform continuity of $w_{f;g,h}$. 

It remains to show uniform and unconditional convergence on compact sets. 
 We first consider the case $g_j = h_j$. Here the terms on the right-hand side are positive, continuous functions, whose partial sums are bounded by the Bessel constant of the system generated by the $(g_j)_{j \in J}$. We already showed that the limit function is continuous, as well, and since the group is metrizable \cite{HeRo}, we may apply Dini's theorem to conclude that the sum converges uniformly on compact sets. 
 
 For the general case, fix $\epsilon >0$ and a compact set $K \subset
 G$. 
 Fix a finite set $J' \subset J$ with the property that, for all $J''
 \supset J'$, it holds $|w_{f:g}(x)-\sum_{j \in J''} w_{f;g,j}(x)| <
 \epsilon/B_h$ for all $x \in K$. 
 
 The Cauchy-Schwartz inequality yields, for all $j \in J$, that 
 \[
  |w_{f;g,h,j}(x)| \le w_{f;g,j}^{1/2}(x) w_{f;h,j}^{1/2}(x).
 \]
 It follows, by a second application of the Cauchy-Schwarz inequality, that 
 \begin{eqnarray*}
  \left|w_{f;g,h}(x) - \sum_{j \in J''} w_{f;g,h,j}(x) \right| & = & \left| \sum_{j \in J \setminus J''} w_{f;g,h,j}(x) \right| \\
 &  \le & \underbrace{\left( \sum_{j \in J \setminus J''} w_{f;g,j} (x) \right)^{1/2}}_{< \epsilon/B_h}    \underbrace{\left( \sum_{j \in J \setminus J''} w_{f;h,j} (x) \right)^{1/2}}_{\le B_h} \\
  & < & \epsilon~,
 \end{eqnarray*}
whenever $x \in K$. This proves uniform and unconditional convergence on compact sets. 
 \end{proof}

\subsection{Almost periodic functions and the unconditionally
  convergence property}
\label{sec:spac-almost-peri}

The significance of almost periodic functions for GSI systems comes from the fact that the function $w_{f;g,h}$ is a sum of trigonometric polynomials. As soon as this sum converges uniformly, $w_{f;g,h}$ is an almost periodic function, and this fact allows to invoke results from the Fourier analysis of such functions, which we now recall. 
Our main sources for this subsection are \cite[Chapter 18]{HeRo} and \cite{Pankov}. 
\begin{definition}
A function  $f \in C_b(G)$ is called \emph{almost periodic} if the set $\{ T_x f : x \in G \} \subset C_b(G)$ is relatively compact with respect to the uniform norm. The space of all almost periodic functions on $G$ is denoted by $\AP(G)$. 
\end{definition}

As elucidated in \cite{Pankov}, almost periodic functions are best
understood in connection with the \emph{Bohr compactification} $G_B$
of the group $G$. This group is constructed by taking the dual group of  $\ghat$, where the latter is endowed with the discrete topology. By construction, $G_B$ is a compact LCA group, and the duality between $G$ and $\ghat$ gives rise to a canonical embedding $i_G: G \to G_B$, an injective, continuous group homomorphism with dense image. Throughout the following, we will identify $G$ with its image under $i_G$, i.e., with a subgroup of $G_B$. If $G$ is noncompact, this image is a proper subset (being noncompact), measurable (being $\sigma$-compact), and therefore of measure zero: Any measurable subgroup of positive measure contains a neighborhood of the neutral element, and is therefore open.

We now have the following characterizations of almost periodic functions. Note that by definition, a trigonometric polynomial is a linear combination of characters. 
\begin{theorem} \label{thm:char_ap}
Let $f \in C_b(G)$. Then the following are equivalent:
\begin{enumerate}[(a)]
  \item $f \in \AP(G)$.
  \item $f$ is the uniform limit of trigonometric polynomials. 
  \item $f$ has a (necessarily unique) continuous extension  $f_B : G_B \to \mathbb{C}$. 
 \end{enumerate}
\end{theorem}

We will call the function $f_B$ from part (c) the \emph{Bohr extension} of $f \in \AP(G)$. Part (c) opens the door to the Fourier analysis of almost periodic functions (and therefore for a proof of (b)), by making Fourier expansions of $f_B$ available for the analysis of $f$. Since $G_B$ is compact, every continuous function on $G_B$ is the continuous limit of trigonometric polynomials. But $G$ and $G_B$ share the same dual $\ghat$ (only the induced topologies are different), hence this approximation result translates to functions in $\AP(G)$. 
In order to compute the Fourier coefficients of $f_B$, we need to integrate over $G_B$, or better, devise an integration process on $G$ that allows to compute these integrals without explicitly passing to the extension $f_B$. This is where the \emph{mean} on $\AP(G)$ comes into play, which is described in the following result, which summarizes Theorems~18.8-18.10 from \cite{HeRo}.
\begin{theorem}
\label{thm_mean}
 Let $G$ denote a second countable LCA group.
 \begin{enumerate}[(a)]
  \item There exists a sequence $(H_n)_{n \in \mathbb{N}}$ of open, relative compact subsets $H_n \subset G$ with $G = \bigcup_{n \in \mathbb{N}} H_n$ and such that, for all $x \in G$, 
  \[
   \lim_{n \to \infty} \frac{\mu_G((x+H_n) \cap (G \setminus H_n))}{\mu_G(H_n)} = 0~.
  \]
\item Let $(H_n)_{n \in \mathbb{N}}$ be a sequence of subsets as in part (a). For any $f \in \AP(G)$, the expression
\[
 \mean{f} = \lim_{n \to \infty} \frac{1}{\mu_G(H_n)} \int_{H_n} f(x) dx
\] is well-defined and finite. 
Furthermore, if $f_B$ denotes the Bohr extension of $f$, then 
\[
 \mean{f} = \int_{G_B} f_B(y) dy~.
\]
\item As a consequence of (b), $\mean{f}$ is independent of the choice of $(H_n)_{n \in \mathbb{N}}$. 
 \end{enumerate}
\end{theorem}

The quantity $\mean{f}$, as defined in Theorem~\ref{thm_mean}, denotes  the \emph{mean} of $f \in \AP(G)$. Given any $\alpha \in \ghat$, we can then define the \emph{Fourier coefficient} of $f$ as 
\[
 \hat{f}(\alpha) = \mean{f \overline{\alpha}}~.
\]
Using the facts that the map $\AP(G) \ni f \mapsto f_B \in C(G_B)$ is injective, and that the duals of $G$ and $G_B$ coincide, standard facts of Fourier analysis on $G_B$ give rise to the following important theorem.
\begin{theorem} Let $f \in \AP(G)$. We then have:
 \begin{enumerate}
  \item[(a)] \emph{Fourier uniqueness}: $f = 0$ if and only if $\hat{f}(\alpha) = 0$, for all $\alpha \in \ghat$.
  \item[(b)] \emph{Plancherel Theorem}: $\mean{|f|^2} = \sum_{\alpha \in \ghat} |\hat{f}(\alpha)|^2$. In particular, only countably many Fourier coefficients are nonzero. 
 \end{enumerate}
\end{theorem}

We remark that $\AP(G)$ under the inner product $\innerprod[AP]{f}{g}=\mean{f \,\overline{g}}$ for
$f,g \in \AP(G)$ is a pre-Hilbert space; its completion is
called the Besicovitch space $B^2(G)$, for which $\seq{\alpha}_{\alpha
  \in \ghat}$ is an orthonormal basis. The completion with respect to the norm
$\mean{\abs{\cdot}}$ gives rise to the Besicovitch space $B^1(G)$.

With these definitions in place, we can now introduce a technical condition that will be of central importance to the following.
\begin{definition} \label{def:UCP}
 Let $(\Gamma_j)_{j \in J}$ denote a system of lattices, and let $(g_j)_{j \in J}$ and $(h_j)_{j \in J}$ denote generating systems in $L^2(G)$. 
 \begin{enumerate}[(a)]
\item 
The GSI systems $(T_\gamma g_j)_{j,\gamma}$ and $(T_\gamma
g_j)_{j,\gamma}$ (or, for short, simply $(g_j)_j$ and $(h_j)_j$)  have the \emph{dual $1$-unconditional convergence property}
     (dual $1$-UCP) if, for all $f \in \cD$, $w_{f;g,h} \in \AP(G)$ \emph{and}
   \begin{equation} \label{eqn:con_wf} w_{f;g,h} = \sum_{j \in J}
     w_{f;g,h,j}
   \end{equation}
   unconditionally with respect to $\mean{|\cdot|}$, i.e., for every $\epsilon>0$
   there exists a finite set $J' \subset J$ such that for all finite set
   $J'' \supset J'$,
   \[
   M\left( \left| w_{f;g,h} - \sum_{j \in J''} w_{f;g,h,j} \right|
   \right) < \epsilon
   \]
 \item  The systems $(g_j)_j$ and $(h_j)_j$ have the \emph{dual $\infty$-UCP} if (\ref{eqn:con_wf}) holds
   with uniform convergence.
 \item If $h_j = g_j$ holds, for all $j \in J$, we  say that the system $(g_j)_{j \in J}$ fulfills $p$-UCP, for $p \in \{1, \infty\}$.
 \end{enumerate}
\end{definition}

Note that the H\"older inequality implies $\mean{|f|^p}^{1/p} \le \mean{|f|^q}^{1/q} \le \| f \|_\infty$, whenever $1 \le p\le q < \infty$. In particular, $\infty$-UCP is the stronger condition.

\begin{remark}
 The $1$-UCP condition can be rephrased as requiring that  
 \[ (w_{f;g,h})_{B} = \sum_{j \in J} (w_{f;g,h,j})_{B} 
 \]
with convergence in $L^1(G_B)$, where we again used the subscript $B$
to denote the (continuous) Bohr extension. This is one of the reasons
why the condition $w_{f;g,h} \in \AP(G)$ is included in the
definition of $1$-UCP. Clearly, checking this part can present a nontrivial obstacle. Note however that in the case $p=\infty$, uniform convergence already implies $w_{f;g,h} \in \AP(G)$. 
  Also, note that for the derivation of \emph{necessary} conditions for dual systems, one departs from the assumption that $w_{f;g,h} \equiv 1$, which is clearly in $\AP(G)$. 
\end{remark}

\begin{remark}
As for local integrability conditions, the $p$-UCP depends on the blind
spot set $E$, which is a closed set $E \subset \ghat$ of measure zero.
Since $\mathcal{D}_E \subset \mathcal{D}_\emptyset$ for any blind spot
set $E$, see \eqref{eq:def-D}, it follows that if $p$-UCP holds for
$E=\emptyset$, it holds for any $E$. However, we
usually only need the $p$-UCP to hold for \emph{some} blind spot set $E$.
\end{remark}

%

The Bessel generator assumption and the $\infty$-UCP  can be seen as regularity
assumptions on $w_{f;g,h}$, e.g., both assumptions separately guarantee that
$w_{f;g,h}$ is continuous. The two assumptions are in general
unrelated. Bessel generators $(g_j)_j$ do not imply
$\infty$-UCP and the $\infty$-UCP does not imply Bessel generators.      
However, the following result shows that the analysis windows $(g_j)_{j \in J}$
and synthesis windows $(h_j)_{j \in J}$ can be  separated  in the
verification of the dual $p$-UCP condition, when
combined with a Bessel assumption. 
\begin{lemma}\label{lem:UCP_single} 
 Assume that we are given lattices $(\Gamma_j)_{j \in J}$, and generating systems $(g_j)_{j \in J}$ and $(h_j)_{j \in J}$.
 \begin{itemize}
  \item[(a)] Suppose that $(g_j)_{j \in J}$ fulfills $\infty$-UCP, and that $(h_j)_{j \in J}$ is a system of Bessel generators. Then $(g_j)_{j \in J}$ and $(h_j)_{j \in J}$ fulfill the dual $\infty$-UCP.  The same result holds with assumptions on $(h_j)_{j \in J}$ and $(g_j)_{j \in J}$ interchanged. 
  \item[(b)] Suppose that $(g_j)_{j \in J}$ fulfills $1$-UCP, and that $(h_j)_{j \in J}$ is a system of Bessel generators such that $w_{f;g,h} \in \AP(G)$.  Then $(g_j)_{j \in J}$, $(h_j)_{j \in J}$ fulfill the dual $1$-UCP. The same result holds with assumptions on $(h_j)_{j \in J}$ and $(g_j)_{j \in J}$ interchanged. 
 \end{itemize}
\end{lemma}

\begin{proof}
 Note that, for any finite set $J' \subset J$, and any $x \in G$,  
 \begin{eqnarray*}
  \left| w_{f;g,h}(x) - \sum_{j \in J'} w_{f;g,h,j}(x) \right| & \le & \sum_{j \in J \setminus J'} \sum_{\gamma \in \Gamma_j} \left| \langle T_x f, T_\gamma g_j \rangle  \langle T_\gamma h_j . T_x f \rangle \right| \\ & \le & \left| \sum_{j \in J \setminus J'} w_{f;g,j}(x) \right|^{1/2} w_{f;h}(x)^{1/2}  \\
  & \le & B_h^{1/2} \| f \| \left| w_{f;g}(x)-\sum_{j \in J'} w_{f;g,j}(x) \right|^{1/2}
  \end{eqnarray*}
  Now assuming that $(g_j)_{j \in J}$ fulfills $\infty$-UCP, yields the desired conclusions about the dual system $(g_j)_{j \in J},(h_j)_{j \in J}$; in particular uniform convergence of the series yields $w_{f;g,h} \in \AP(G)$.  Since $w_{f;g,h} = \overline{w_{f;h,g}}$, the second statement of (a) follows. 
  In the case of $1$-UCP, H\"older's inequality implies 
  \[ M \left(| w_{f;g}-\sum_{j \in J'} w_{f;g,j}|^{1/2}\right) \le M\left(| w_{f;g}-\sum_{j \in J'} w_{f;g,j}|\right)^{1/2} ~,\] hence part (b) follows in the same way as part (a). 
\end{proof}

\begin{proposition} \label{lem:Four-coeff-wf}
Assume that we are given lattices $(\Gamma_j)_{j \in J}$, and
generating systems $(g_j)_{j \in J}$ and $(h_j)_{j \in J}$ that
fulfill the dual $1$-UCP.  Then, for all $f \in \cD$ and  all
$\alpha \in \ghat$, we have
   \begin{equation}     \label{eq:Fourier-coeff-wf} 
   \widehat{w_{f;g,h}}(\alpha) = \sum_{j \in J}
   \widehat{w_{f;g,h,j}}(\alpha) 
  \end{equation}
  with absolute convergence. Hence, the
generalized Fourier coefficients of $w_{f;g,h} = \sum_{\alpha \in \ghat}
c_\alpha \,
\alpha$  are  given by
\begin{equation}
  \label{eq:gen-Fourier-coeff-wf}
c_\alpha \equiv \widehat{w_{f;g,h}}(\alpha) = 
\begin{cases}
  \sum_{j \in J: \alpha \in \Gamma_{j}^\bot} d_{j,\alpha}  &
  \text{for }\alpha \in \bigcup_{j \in J} \Gamma_j^\bot, \\ 
 0 & \text{otherwise}.
\end{cases}
\end{equation}
\end{proposition}
\begin{proof}
Recall that $1$-UCP entails $L^1$-convergence of the Bohr extensions,
and that the Fourier coefficients of any function in
$\AP(G)$ coincide with the coefficients of its Bohr
extension. Since Fourier coefficients are continuous linear
functionals with respect to the $1$-norms, equation
(\ref{eq:Fourier-coeff-wf}) follows. The last statement of the theorem
is just a reformulation of \eqref{eq:Fourier-coeff-wf} using Lemma~\ref{lem:wfj-convergence}.
\end{proof}

\begin{remark}
  \label{rem:w_f-under-infty-UCP}
  Under the dual $\infty$-UCP assumption in
  Proposition~\ref{lem:Four-coeff-wf} in place for the dual $1$-UCP,
  the conclusions of Proposition~\ref{lem:Four-coeff-wf} still hold
  true, in particular, that $w_{f;g,h}$ is a continuous and almost
  periodic function that agrees pointwise with its generalized Fourier
  series. 
\end{remark}

The importance of the function $w_{f;g,h}$ and its generalized Fourier
series in
Proposition~\ref{lem:Four-coeff-wf} is that it encodes frame
theoretical properties of GSI systems. E.g., under the UCP assumption,
a GSI system \gsi\ is a
frame with optimal bounds $A^\dagger$ and $B^\dagger$ if and only if 
 \begin{equation}
  \label{eq:wf-opt-bounds-B}
   B^\dagger = \sup_{f \in \cD,\norm{f}=1} w_{f;g}(0) =
   \sup_{f \in \cD,\norm{f}=1} \max_{x \in G} \sum_{j \in J }w_{f;g,j}(x) < \infty  
\end{equation}
and 
 \begin{equation}
  \label{eq:wf-opt-bounds-A}
   A^\dagger = \inf_{f \in \cD,\norm{f}=1} w_{f;g}(0) =
   \inf_{f \in \cD,\norm{f}=1} \min_{x \in G} \sum_{j \in J } w_{f;g,j}(x)   > 0.
\end{equation}
Fourier analysis of $w_{f;g}$ was recently employed in
\eqref{eq:wf-opt-bounds-B} and \eqref{eq:wf-opt-bounds-A} to obtain
new sufficiently conditions for the frame property of GSI systems \cite{1704.06510}.
Moreover, whenever the mixed frame operator $S_{g,h}=\sum_{j,\gamma}
\langle \cdot, T_\gamma g_j \rangle T_\gamma h_j$ on $L^2(G)$ of \gsi\
and $\gsi[h]$ is well-defined, but not
necessarily bounded, the discussions in the next section
yield the representation, under appropriate convergence assumptions,
\[ 
  \innerprodbig{\widehat{S_{g,h}f}}{\hat{f}} = \sum_{\alpha \in \bigcup_{j \in
    J} \Gamma_j^\perp} \innerprodbig{T_{\alpha}M_{t_\alpha}
\hat{f}}{\hat{f}} \quad  \text{for each $f \in \cD$},
\]
where $M_{t_\alpha}$ denotes the multiplication operator by $t_\alpha = \sum_{j\in J \, : \, \alpha\in\LG_{\! j}^{\perp}} 
 \frac{1}{\covol(\Gamma_j)} \,  \overline{\hat{g}_j(\cdot)}
 \hat{h}_j(\cdot+\alpha )$.

\subsection{Characterizing equations for dual and tight GSI frames}
\label{sec:char-equat-dual}

The following theorem exploits the fact that convergence of the sum
\begin{equation}
 \label{sum:wf}
 w_{f;g,h} = \sum_{j \in J} w_{f;g,h,j}
\end{equation}
in the proper sense results in expressions for the Fourier coefficients
\[
   d_{j,\alpha} = \frac{1}{\covol(\Gamma_j)}\int_{\ghat} \hat{f}(\omega) \overline{\hat{g}_j(\omega) \hat{f}(\omega+\alpha)} \hat{h}_j(\omega+\alpha) d\omega~.
 \]

\begin{theorem} \label{thm:meta_convergence_AP} 
Suppose that \gsi and \gsi[h] are Bessel families fulfilling the dual $1$-UCP. Then the following are equivalent: 
\begin{enumerate}[(i)]
 \item \gsi and \gsi[h] are dual frames for $L^2(G)$, \label{item:dual-char-1}
 \item for each $\alpha \in \bigcup_{j\in J} \LG_{\! j}^{\perp}$ we have
\begin{equation}
 \label{eq:t-alpha} t_{\alpha}(\omega) := \sum_{j\in J \, : \, \alpha\in\LG_{\! j}^{\perp}} 
 \frac{1}{\covol(\Gamma_j)} \,  \overline{\hat{g}_j(\omega)} \hat{h}_j(\omega+\alpha )  =
\delta_{\alpha,0} \quad a.e.\ \omega \in \ghat 
\end{equation}
with absolute convergence. \label{item:dual-char-2}
\end{enumerate}
\end{theorem}
\begin{proof}
\eqref{item:dual-char-1} $\Rightarrow$ \eqref{item:dual-char-2}:
Fix $f \in \cD$. The dual frame assumption yields $w_{f;g,h} \equiv \norm{f}^2$, which entails by Fourier uniqueness that,
for all $\alpha \in \ghat$, 
\[
\norm{f}^2 \delta_{0,\alpha}= \langle w_{f;g,h},\alpha  \rangle_{AP} = 
 \sum_{j \in J: \alpha \in \Gamma_{j}^\bot} d_{j,\alpha} ~,
\] 
with unconditional convergence, where we have used Proposition~\ref{lem:Four-coeff-wf}.

For the derivation of \eqref{eq:t-alpha} from this fact, we adopt the strategy used in the proof of \cite[Theorem 3.4]{JakobsenReproducing2014}.
Fix $\alpha \in \bigcup_{j \in J} \Gamma_{j}^\bot$, and define
\[
 t_\alpha(\omega) =  \sum_{j \in J: \alpha \in \Gamma_j^\bot}  \frac{1}{\covol(\Gamma_j)}  \overline{\hat{g}_j(\omega)} \hat{h}_j(\omega + \alpha).
\]
Note that the right-hand side actually converges absolutely by the
following chain of inequalities:
\begin{multline}\label{eq:t-alpha-abs-conv}
\sum_{j \in J: \alpha \in \Gamma_j^\bot}
  \frac{1}{\covol(\Gamma_j)} \left| \overline{\hat{g}_j(\omega)} \hat{h}_j(\omega + \alpha) \right| 
 \le  \sum_{j \in J}  \frac{1}{\covol(\Gamma_j)} \left| \overline{\hat{g}_j(\omega)} \hat{h}_j(\omega + \alpha) \right|\\
 \le  \left( \sum_{j \in J} \frac{1}{\covol(\Gamma_j)} \left| {\hat{g}_j(\omega)} \right|^2 \right)^{1/2} 
\left( \sum_{j \in J} \frac{1}{\covol(\Gamma_j)} \left| {\hat{h}_j(\omega+\alpha)} \right|^2 \right)^{1/2}  \le  B_g^{1/2} B_h^{1/2},
\end{multline}
where the last inequality is due to Lemma~\ref{lem:calderon_upper},
and $B_g$ and $B_h$ are the Bessel constants associated to  $(g_j)_{j
  \in J}$ and $(h_j)_{j \in J}$, respectively. 
This shows also that 
\[
 \sum_{j \in J: \alpha \in \Gamma_j^\bot} \int_{\ghat} \left| \hat{f}(\omega) \overline{\hat{g}_j(\omega)} \overline{\hat{f}(\omega+\alpha)} \hat{h}_j(\omega + \alpha)\right| d\omega < \infty .
\]
Hence the multiplication operator $M_{\overline{t_\alpha}} : L^2(\widehat{G}) \to L^2(\widehat{G})$, $f  \mapsto f \cdot \overline{t_\alpha}$ is well-defined and bounded. For all $f \in \cD$ we have the relation
\[
 \langle \hat{f}, M_{\overline{t_\alpha}} T_{-\alpha} \hat{f}
 \rangle = \sum_{j \in J: \alpha \in \Gamma_{j}^\bot} d_{j,\alpha}  = \delta_{0,\alpha}\innerprods{\hat{f}}{\hat{f}}~,
\] and since $\cD$ is dense, this implies
\[ 
  M_{\overline{t_\alpha}} T_{-\alpha} =
  \begin{cases}
    I & \alpha =0, \\
    0 & \alpha \notin \bigcup_{j\in J} \LG_{\! j}^{\perp} \setminus \{0\},
  \end{cases}
\]
which in turn yields \eqref{eq:t-alpha}. 

\eqref{item:dual-char-2} $\Rightarrow$ \eqref{item:dual-char-1}: We have
\[
 \sum_{j \in J: \alpha \in \Gamma_{j}^\bot} d_{j,\alpha}  = \langle \hat{f}, M_{\overline{t_\alpha}} T_{-\alpha} \hat{f}
 \rangle =  \delta_{0,\alpha}\innerprods{\hat{f}}{\hat{f}}~,
\]  for all $f \in \cD$, with the last equation provided by the assumption on the $t_\alpha$. Hence, by Proposition \ref{lem:Four-coeff-wf},
the Fourier coefficients of $w_{f;g,h} \in \AP(G)$ coincide with the Fourier coefficients of the constant function. Hence the Bohr extension of $w_{f;g,h}$ is constant, and consequently, so is $w_{f;g,h}$. 
\end{proof}

\begin{remark}  \label{rem:ucp_vs_lic}
 Theorem~\ref{thm:meta_convergence_AP} is indeed a generalization of \cite[Theorem
 3.4]{JakobsenReproducing2014}. The cited result is established under
 the assumption that the so-called \emph{dual $\alpha$-local
   integrability condition (dual $\alpha$-LIC)} holds. This condition involves the coefficients
 $(c_{j,\alpha})_{j,\alpha}$ associated to each $f \in \cD$ via 
 \[
 c_{j,\alpha} =   \frac{1}{\covol(\Gamma_j)}\int_{\ghat} \left| \hat{f}(\omega) \hat{g}_j(\omega) \hat{f}(\omega+\alpha) \hat{h}_j(\omega+\alpha) \right| d\omega.
 \] 
A comparison with the definition of $d_{j,\alpha}$ shows that $c_{j,\alpha}$ is obtained by taking the absolute value of the integrand in the definition of $d_{j,\alpha}$, in particular, 
 \[
  |d_{j,\alpha}| \le c_{j,\alpha}~.
 \] 
Now the dual $\alpha$-LIC  amounts to the requirement that
$\sum_{j,\alpha} \abs{c_{j,\alpha}} < \infty$ for every $f \in \cD$.
Since characters $\alpha \in \ghat$ are bounded, the $\alpha$-LIC implies in fact that
 \[
  w_{f;g,h} = \sum_{j \in J} \sum_{\alpha \in
    \Gamma_j^\bot} d_{j,\alpha} \, \alpha~,
 \] with the right hand side converging \emph{unconditionally and
   uniformly}. Thus dual $\infty$-UCP is guaranteed, and we may apply Theorem~\ref{thm:meta_convergence_AP}  to recover \cite[Theorem 3.4]{JakobsenReproducing2014}. 

 Note also that this argument goes through under the assumption that
 $\sum_{j,\alpha} |d_{j,\alpha}| < \infty$. This condition could be
 strictly weaker than the dual $\alpha$-LIC because of possible cancellations inside the integrals defining the $d_{j,\alpha}$ that could entail $|d_{j,\alpha}| \ll c_{j,\alpha}$. However, we are not aware of examples where this observation pays off. 
Finally, we remark that the classical LIC \cite{MR1916862,MR2283810} amounts to the requirement that
$\sum_{j,\alpha} \abs{\tilde{c}_{j,\alpha}} < \infty$ for every $f \in
\cD$, where 
 \[
 \tilde{c}_{j,\alpha} =   \frac{1}{\covol(\Gamma_j)}\int_{\ghat}
 \absbig{ \hat{f}(\omega) \hat{f}(\omega+\alpha)} \abs{\hat{g}_j(\omega)}^2  d\omega, 
 \] 
which implies the $\alpha$-LIC and therefore also the $\infty$-UCP.
\end{remark}

In the case of tight frames, the Bessel assumption in
Theorem~\ref{thm:meta_convergence_AP} is not necessary.
\begin{theorem} \label{thm:meta_convergence_AP_tight} 
Suppose that \gsi\ fulfills the $1$-UCP. Then the following are equivalent: 
\begin{enumerate}[(i)]
 \item \gsi\ is a Parseval frames for $L^2(G)$, \label{item:char-parseval-1}
 \item for each $\alpha \in \bigcup_{j\in J} \LG_{\! j}^{\perp}$ we
   have 
\begin{equation}
 \label{eq:t-alpha_tight} \sum_{j\in J \, : \, \alpha\in\LG_{\! j}^{\perp}} 
 \frac{1}{\covol(\Gamma_j)} \,  \overline{\hat{g}_j(\omega)} \hat{g}_j(\omega+\alpha )  =
\delta_{\alpha,0} \quad a.e.\ \omega \in \ghat .
\end{equation} \label{item:char-parseval-2}
\end{enumerate}
\end{theorem}
\begin{proof}
  We claim that if either \eqref{item:char-parseval-1}  or
  \eqref{item:char-parseval-2} holds, the convergence in
  \eqref{eq:t-alpha_tight} is absolute. The proof of this claim for 
  \eqref{item:char-parseval-1} is
  clear from the chain of inequalities \eqref{eq:t-alpha-abs-conv}
  with $h_j=g_j$, $j \in J$, since
  \gsi\ is a Bessel system by assumption. On the other hand, if
  \eqref{item:char-parseval-2} holds, then for $\alpha=0$, we have
  $\sum_{j \in J}\frac{1}{\covol(\Gamma_j)} \,  \abs{\hat{g}_j(\omega)}^2=1$. By
  computations as in \eqref{eq:t-alpha-abs-conv}, this proves the
  claim.

  The rest of the proof follows the proof of Theorem~\ref{thm:meta_convergence_AP}.
\end{proof}

 The following example was discovered by Bownik and
 Rzeszotnik~\cite{MR2066821} to show that the Calder\'on sum for
 Parseval GSI frames is not necessarily equal to one.
 The example was also the first
 construction to show that the
 $t_\alpha$-equations \eqref{eq:t-alpha_tight} do not characterize Parseval GSI frames without
 some regularity assumption on $\Gamma_j$ and $g_j$, e.g., the LIC. In the context
 of this paper, the example serves as an illustration 
 that $1$-UCP allows finer distinctions than LIC. It is constructed in $\ell^2(\mathbb{Z})$, but can be easily transferred to $L^2(\mathbb{R})$. 
 
\begin{example} \label{ex:UCP_LIC}
  Let $G = \Z$ and, for each $N=2,3,\dots$, write $\Z$ as a disjoint
  union:
\[
 \Z = \bigcupdot_{j \in \N} \tau_j + N^j \Z, 
\]
where $\tau_1=0$ and $\tau_j$, $j \ge 2$, are chosen inductively as the smallest $t \in \Z$ in absolute value
satisfying 
\[
\left(\bigcupdot_{i=1}^{j-1} \left(\tau_j + N^j \Z \right)\right) \cap \left(t +
  N^j \Z\right) = \emptyset.
\]
It case $t$ and $-t$ both are minimizers, we pick $\tau_j$ to be
positive.

For $j\in J = \N$, let $\Gamma_j=N^j\Z$ and $g_j=\delta_{\tau_j}$, where
$\delta_k$ denotes the sequence with $\delta_k(k)=1$ and
$\delta_k(\ell)=0$ for $\ell\neq k \in \Z$. The GSI system $\gsi$ is
an orthonormal basis for $\ell^2(\Z)$ since it is a reordering
of the canonical orthonormal basis $\seq{\delta_k}_{k \in \Z}$. Bownik and Rzeszotnik~\cite{MR2066821} show that
$t_0(\omega)= \frac{1}{N-1}$ and that the GSI systems do not satisfy the
LIC. For $N\ge 3$ this shows that the local integrability condition of
\cite[Theorem 2.1]{MR1916862} cannot be removed. Kutyniok and Labate~\cite{MR2283810}
used the example with $N=2$ to show that Parseval frames need not
satisfy the LIC.

In \cite{JakobsenReproducing2014} it was noted that the
characterizing equations~\eqref{eq:t-alpha_tight} are satisfied for $N=2$; to be precise, the
example in \cite{JakobsenReproducing2014} is slightly different from
the present one, but
the verification of the characterizing equations for our example is very
similar, hence we leave out the details. Since GSI systems can satisfy
the characterizing equations, but not the local integrability conditions
nor the weaker $\alpha$-LIC, it leaves room for an improvement of the
results in both \cite{MR1916862} and
\cite{JakobsenReproducing2014}. Hence, none of the known results on
characterizing $t_\alpha$-equations can be applied for the case
$N=2$. However, we will now show that
Theorem~\ref{thm:meta_convergence_AP} indeed can capture this phenomenon: For $N=2$ the 1-UCP holds,
while it fails for $N\ge 3$. This is the desired conclusion
as only the case $N=2$ satisfies the characterizing equations.

A first indication of the striking difference between $N=2$ and $N\ge
3$ comes from the growth rate of $\tau_j$. For $N=2$ we find by
induction that $\tau_j=-\frac13(-2)^j+\frac13$, while
$\abs{\tau_j}$ grows linearly as $\tfrac{j}{2} \le \abs{\tau_j} \le j$
 for $N\ge 3$. 

Let $N \ge 2, N \in \N$ be given.
Since $\gsi$ is an orthonormal basis, we have that $w_{f;g}(x)=\sum_j w_{f; g, j}(x)=\norm{f}^2$ for
$a.e.\ x \in \Z$.  For simplicity, assume
$f \in \mathcal{D}_\emptyset$ is normalized, that is, $\norm{f}^2=1$
and $E=\emptyset$. For each $j \in \N$, 
\[
w_{f; g, j}(x)=  \sum_{\gamma \in \LG_j} \absBig{\innerprodbig{T_x f }{T_\gamma g_j}\,}^2 =  \sum_{\ell \in H_{x,j}}
\abs{f(\ell)}^2, 
\]
where $H_{x,j}=x + \tau_j + N^j \mathbb{Z}$. Hence, $w_{f;g,j}(x)$ is the value of the squared norm
  of the orthogonal projection of $T_xf$ onto
  $\overline{\Span\setprop{\delta_k}{k \in \LG_j}}$. 
  
We first prove that $\infty$-UCP is not satisfied for any choice of $N$. For this purpose, let $J' \subset N$ be any finite subset. 
Let $x \in \mathbb{Z}$ be such that $m-x \not\in \bigcup_{j \in J}  \tau_j + N^j \mathbb{Z}$; note that $x$ exists by construction of the $\tau_j$.
It then follows that $\sum_{j \in J} w_{f;g,j} (x) \le 1-|f(m)|^2$, or 
\[
 w_{f;g}(x) - \sum_{j \in J}w_{f;g,j}(x) \ge  |f(m)|^2~, 
\] and thus $\| w_{f;g}- \sum_{j \in J}w_{f;g,j} \|_\infty \ge
|f(n)|^2$. Hence, $\infty$-UCP does not hold for $E=\emptyset$. Note
that allowing a general blind spot set $E$, i.e., a closed subset of
$\ghat = \mathbb{T}$ of measure zero, does not change this conclusion
as $\cD$ is a non-trivial, translation-invariant subspace of $\ell^2(\Z)$.

Let us next consider $1$-UCP. Note that $w_{f;g}=1$ which is indeed
almost periodic. We first compute the mean of $w_{f;g,j}$. 
Note that $w_{f;g,j}$ is $N^j$-periodic, and that the mean of a periodic function is just the average over one period. Hence we get
\[
\mean{w_{f;g,j}} = \frac{1}{N^j}  \sum_{x = 0}^{N^{j}-1} w_{f;g,j}(x) =\frac{1}{N^j} \underbrace{\sum_{x=0}^{N^{j}-1} \sum_{m \in x+\tau_j + N^j \mathbb{Z}} |f(m)|^2}_{= \| f \|^2} \\
=  \frac{1}{N^j} ~.
\]
Using linearity of the mean, we get for any finite set $J' \subset \mathbb{N}$  
\begin{equation}
 \mean{1-\sum_{j \in J'} w_{f;g,j}} = 1 - \sum_{j \in J'} N^{-j}~.
\end{equation}
In particular, for $N=2$ convergence of the geometric series yields 
\[ \mean{1-\sum_{j \in J} w_{f;g,j} } \to 0 \] unconditionally, and
thus $1$-UCP holds. For $N>2$, however, 
\[ \mean{1-\sum_{j \in J} w_{f;g,j}} \ge 1- \sum_{j \in \N} N^{-j}  = \frac{N-2}{N-1} \]
shows that $1$-UCP is violated. 

  \end{example}

\subsection{Necessary conditions for the frame property}
\label{sec:necess-cond}

From the characterizing equations in Theorem~\ref{thm:meta_convergence_AP} we can derive a necessary condition
for the frame property of a GSI system in terms of the Calder\'on
sum. The condition can be seen as a quantitative version of the fact
that the Fourier supports of the generators need to cover $\ghat$.  

\begin{theorem} \label{thm:a-lic-calderon-bounds} Suppose that \gsi\ is a frame for
  $L^2(G)$ satisfying the 
  $1$-UCP. Then 
  \begin{equation}
    \label{eq:calderon-bounded-below}
    A_g \le \sum_{j\in J}  \frac{1}{\covol(\Gamma_{\!j})} \abs{\hat{g}_j(\omega)}^2   \quad \text{for a.e. }  \omega \in \ghat.
  \end{equation}
\end{theorem}

  \begin{proof} 
Let $f\in \cD$.
By assumption, $w_{f;g} \in \AP(G)$, and its mean is equal to the constant term of
its Fourier series by Definition~\ref{def:UCP}, that is,
\[
 \mean{w_{f;g}} = 
\innerprod[AP]{w_{f;g}}{\mathbf{1}_{G}} = \sum_{j \in J} d_{j,0} = \int_{\ghat} \abssmall{\hat{f}(\omega)}^2\, t_{0}(\omega) \,
    d\omega, 
\]
where we have used~\eqref{eq:Fourier-coeff-wf}.

The frame inequality implies that 
$A_g\norm{f}^2 \le w_{f;g}(x)$ for all  $x \in G$. 
Since $w_{f;g} - A_g \norm{f}^2 \ge 0$ a.e., it follows that
\[
\mean{w_{f;g}}\ge \mean{A_g \norm{f}^2} = \mean{\mathbf{1}_{G}} A_g \norm{f}^2 = A_g
\norm{f}^2.
\]
Hence, we arrive at
 \begin{equation}
  A_g\norm{f}^2 \le \int_{\ghat} \abssmall{\hat{f}(\omega)}^2\, t_0(\omega) \,
      d\omega \label{eq:bound-of-t1}. 
 \end{equation}

This, in turn, implies that $A_g \le t_0(\omega)$ for a.e.\ $\omega \in
\ghat$. To see this, assume towards a contradiction that
$t_{0}(\omega)<A_g$ for $\omega \in F$, where $F$ is of positive
measure. Let $\hat{f}(\omega)=\mathbf{1}_{F}$. Then 
$ \int_{\ghat} \abssmall{\hat{f}(\omega)}^2\, t_0(\omega) \,
      d\omega = \int_{F} t_0(\omega) \,
      d\omega < A_g \norm{f}^2 $, which contradicts
      \eqref{eq:bound-of-t1}. 
  \end{proof}

  \begin{remark}
    Note that Theorem~\ref{thm:a-lic-calderon-bounds} is a
    generalization of results in \cite{MR3627474,MR1199539}. The argument
    in the proof of Theorem~\ref{thm:a-lic-calderon-bounds} can also be
    used to prove an upper bound, but this bound holds without any
    LIC/UCP assumptions by Lemma~\ref{lem:calderon_upper}. We refer to
    Section~\ref{sec:appl-extens-} for applications to wavelet systems
    and generalizations of Theorem~\ref{thm:a-lic-calderon-bounds}.
  \end{remark}

The following theorem substantiates the intuition on the role of
bandwidth for the existence of generators. It proves that infinite
bandwidth is necessary for the existence of frame generators $(g_j)_{j
\in J}$ in non-discrete spaces. 
\begin{theorem}
\label{thm:density-GTI}
Suppose that \gsi is a frame for
  $L^2(G)$ which satisfies the $1$-UCP. Then  
\[ 
BW(\mathcal{G}) \equiv \sum_{j\in J}  \frac{1}{\covol(\Gamma_{\!j})}
\ge 
\frac{A_g}{B_g} \, \mu_{\ghat}(\ghat).
\]
In particular, if $G$ is non-discrete, then $BW(\mathcal{G})=\infty$.
\end{theorem}

\begin{proof}
  By integrating the lower bound in
  \eqref{eq:calderon-bounded-below} over $\ghat$, we obtain
 \begin{equation}
\int_{\ghat} A_g \, d\omega \le  \sum_{j\in J}
\frac{1}{\covol(\Gamma_{\!j})} \norm{\hat{g}_j}^2. 
\label{eq:int-over-lower}
\end{equation}
From frame theory we know that a countable Bessel family $(\eta_i)_{i \in I}$ with Bessel bound $B$ in a Hilbert
space satisfies the norm bound $\norm{\eta_i}^2 \le B$ for all $i
\in I$. In our
settings, using isometry of translations and the Plancherel theorem, this fact yields $\norm{\hat{g}_j}^2 \le B_g$ which, combined with
\eqref{eq:int-over-lower}, proves the
sought inequality.  
Finally, if $G$ is
  non-discrete, the dual group $\ghat$ is non-compact, hence $\int_{\ghat} A_g d\omega$
  is infinite. 
\end{proof}

The following result notes a further basic fact: Lattice systems that generate a frame must be infinite, if $G$ is non-discrete. 
\begin{corollary} \label{cor:finite_ls} Assume that $G$ is non-discrete. 
 Let $(\Gamma_j)_{j \in J}$ denote a finite system of lattices. Then, for every system of generators $(g_j)_{j \in J}$, the associated GSI system does not possess a lower frame bound. 
\end{corollary}
\begin{proof}
Since $G$ is non-discrete, $\ghat$ is non-compact, and therefore has infinite Haar measure. 
A GSI system $\gsi$ with  a finite system of lattices, i.e.,
$\card{J}<\infty$, obviously satisfies
$\infty$-UCP and therefore $1$-UCP.  
 Let $C=\max_{j \in J}\frac{1}{\covol{\Gamma_j}}$. Then
 $BW(\mathcal{G}) \le  C \, (\card{J}) < \infty$. The \gsi system cannot be a
 frame by Theorem~\ref{thm:density-GTI}.
\end{proof}

We end this subsection by remarking that, for discrete LCA groups $G$, the above discussions
yield the following additional necessary condition for the Bessel property:
\[ 
 \sum_{j\in J}
\frac{1}{\covol(\Gamma_{\!j})} \norm{\hat{g}_j}^2 \le  B_g \, \mu_{\ghat}(\ghat).
\] 

\subsection{Sufficient conditions for the frame property}
\label{sec:necess-suff-cond}


The next theorem builds on the intuition that motivated the
introduction of our notion of bandwidth. Note in particular that the
conditions of the theorem can only be fulfilled if $BW(\mathcal{G})
\ge \mu_{\ghat}(\ghat)$: Condition~\eqref{item:Kj-1} implies that $K_j$ must be contained in a fundamental domain modulo $\Gamma_j^\bot$, and then \eqref{item:Kj-2} implies that $\ghat$ can be covered by fundamental domains mod $\Gamma_j^\bot$, as $j$ runs through $J$. The latter can only hold if the measures of these domains at least sum up $\mu_{\ghat}(\ghat)$.
\begin{theorem} \label{thm:shannon}
 Let $\mathcal{G} = (\Gamma_j)_{j \in J}$ denote a family of lattices. Assume that there exist Borel sets $K_j \subset \ghat$, for $j \in J$, fulfilling the following two properties:
 \begin{enumerate}[(i)]
  \item $\mu_{\ghat}(K_j \cap \gamma + K_j) =0$ for all $\gamma
    \in \Gamma_j^\bot \setminus \{ 0 \}$ and for all  $j \in J$,  \label{item:Kj-1}
  \item $\mu_{\ghat}(\ghat \setminus \bigcup_{j \in J}
    K_j) = 0$.  \label{item:Kj-2}
 \end{enumerate}
Then there exists a family $(g_j)_{j \in J}$ such that the associated
GSI system is a Parseval frame. In addition, the system can be chosen
to fulfill the relations
\begin{equation} \label{eqn:ntf_disj_1}
\forall j \in J \forall \alpha \in \Gamma_j^\bot~:~ \hat{g}_j(\omega)
\hat{g}_{j}(\omega+\alpha) = 0 \quad (\text{a.e. } \omega)~,
\end{equation} and
\begin{equation} \label{eqn:ntf_disj_2}
 \forall j_1,j_2 \in J \mbox{ with } j_1 \not= j_2~:~\hat{g}_{j_1}(\omega) \hat{g}_{j_2}(\omega) = 0  \quad(\text{a.e. } \omega)~.
\end{equation}
If, in addition to (i) and (ii),  the sets $\{K_j\}_{j \in J}$ fulfill $\mu_{\widehat{G}}(K_j \cap K_j') = 0$, for $j \not= j'$, as well as $\mu_{\widehat{G}}(K_j) = \frac{1}{\covol(\Gamma_j)}$, there exist orthonormal basis generators with these properties. 
\end{theorem}
\begin{proof}
Without loss of generality, we can assume either $J = \N$ or $J = \{ 1, 2, \ldots, N\}$. 
 For each $j$, pick a fundamental domain $K_{j}^1$ of $\Gamma_j^\bot$
 which satisfies $K_j \subset K_{j}^1$ mod $\Gamma_j^\bot$, and define the function $h_j$ by 
$\hat{h}_j = \covol(\Gamma_j)^{1/2} \mathbf{1}_{K_j^1}$. Then $(T_\gamma h_j)_{\gamma \in \Gamma_j}$ is an orthonormal basis of the closed subspace
 \[
  \mathcal{H}_j^1 = \{ f  \in L^2(G) : \hat{f} \cdot \mathbf{1}_{K_j^1} = \hat{f} \}
 \] by Kluvanek's Theorem \cite{Kluvanek}. 
Next define, for $j \ge 1$, 
\[
 K_j^2 = K_j \setminus \bigcup_{l <j} K_l~.
\]
Then $(K_j^2)_{j \in J}$ is a disjoint covering of $\widehat{G}$, and if we define 
 \[
  \mathcal{H}_j^2 = \{ f  \in L^2(G) : \hat{f} \cdot \mathbf{1}_{K_j^2} = \hat{f} \}~,
 \] we obtain $L^2(G) = \bigoplus_{j} \mathcal{H}_j^2$. Furthermore, for any given $j$, the function $g_j$, defined by 
 $\hat{g}_j = \covol(\Gamma_j)^{1/2} \mathbf{1}_{K_j^2}$ is the projection of $h_j$ into $\mathcal{H}_j^2$. Since this 
 projection commutes with translations, one gets that the associated shift-invariant system $(T_\gamma g_j)_{\gamma \in \Gamma_j} \subset \mathcal{H}_j^2$ is the image of an orthonormal basis
 under the projection onto the subspace $\mathcal{H}_j^2$, and thus a
 Parseval frame of that subspace. Finally, taking the union over
 Parseval frames of an orthogonal sequence of subspaces spanning the whole space yields a Parseval frame of the latter. 
 
 In the case where the $K_j$ fulfill  $\mu_{\widehat{G}}(K_j \cap
 K_j') = 0$, for $j \not= j'$ and $\mu_{\widehat{G}}(K_j) =
 1/\covol(\Gamma_j)$, it follows that $K_j$ and $K_j^2$ only
 differ by a set of measure zero, and the system $(T_\gamma
 g_j)_{\gamma \in \Gamma_j}$ is an orthonormal basis of
 $\mathcal{H}_j^2$. Hence the full system is an orthonormal basis of $L^2(G)$. 
\end{proof}

If the underlying group is $G = \mathbb{R}$, we can now formulate the
following characterization of existence of frame generators. 
\begin{corollary}
\label{cor:bw-inf-existence-in-R}  Suppose that $\mathcal{G} = (\Gamma_j)_{j \in J}$ is a family of lattices in $\mathbb{R}$. Then the following are equivalent:
  \begin{enumerate}[(i)]
   \item There  exists a system $(g_j)_{j \in J}$ of frame generators
     satisfying the LIC-condition. \label{item:R-LIC}
   \item There exists a system $(g_j)_{j \in J}$ of frame generators
     satisfying the $1$-UCP condition. \label{item:R-UCP}
   \item  $BW(\mathcal{G}) = \infty$. \label{item:R-BW}
  \end{enumerate}
\end{corollary}
\begin{proof}
The implication \eqref{item:R-LIC} $\Rightarrow$ \eqref{item:R-UCP} is
clear by Remark \ref{rem:ucp_vs_lic}. Implication \eqref{item:R-UCP} $\Rightarrow$ \eqref{item:R-BW} is provided by Theorem~\ref{thm:density-GTI}. 
 Finally, if $BW(\mathcal{G}) = \infty$, we use
 Theorem~\ref{thm:shannon} to construct generators for
 $\mathcal{G}$. Given any $f \in \cD$, we use
 (\ref{eqn:ntf_disj_1}) and the construction of the $\hat{g}_j$ to verify LIC via 
 \begin{align*}
\sum_{j,\alpha}  |c_{j,\alpha}| &=  \sum_{j,\alpha} \frac{1}{\covol(\Gamma_j)} \int_{\R^1} \left|\hat{f}(\omega) \hat{g}_j(\omega) \hat{g}_j(\omega+\alpha) \hat{f}(\omega+\alpha)\right| d\omega \\
&=  \sum_{j} \frac{1}{\covol(\Gamma_j)} \int_{\R^1}
  |\hat{f}(\omega)|^2  |g_j(\omega)|^2  d\omega \qquad (\text{by } \eqref{eqn:ntf_disj_1}) \\
&=   \int_{\R^1} |\hat{f}(\omega)|^2   \underbrace{\sum_{j} \frac{1}{\covol(\Gamma_j)}|g_j(\omega)|^2}_{\equiv 1} d\omega \\
&=  \| f \|^2~.
 \end{align*} 
 \end{proof}

We remark that the equivalences in
Corollary~\ref{cor:bw-inf-existence-in-R} are false without the LIC/UCP
assumption. This follows from Theorem~\ref{thm:ONB_fbw}, proved in Section~\ref{sec:system-bandwidth}.

The next result describes classes of lattice systems in $\R^n$ for which the intuition from the one-dimensional case remains valid. 
 Example~\ref{ex:counter_inf_bw} in Section~\ref{sec:indep-latt-ucp} shows that the assumption on the singular values cannot be dropped. 
 
\begin{proposition} \label{prop:near_iso}
 Assume that the system $\mathcal{G} = (C_j \Z^n)_{j \in J}$ of lattices in $\mathbb{R}^n$ has the property that for all $j \in J$, the quotient of maximal singular value of $C_j$, divided by the minimal singular value, is bounded by a constant. Then $BW(\mathcal{G}) = \infty$ implies the existence of a family of tight frame generators.
 \end{proposition}
 \begin{proof}
 Fix $j \in J$, and let $\sigma_{\mathrm{min}}(j)$ and $\sigma_{\mathrm{max}}(j)$
 denote the minimal and maximal singular value of $C_j^{-T}$, respectively. By the assumption on the family, we have
 \begin{equation} \label{eqn:aniso}
 \frac{ \sigma_{\mathrm{max}}(j)}{\sigma_{\mathrm{min}}(j)} \le K~,
 \end{equation}
 for $K>0$ fixed. We let $C_j^{-T} = U D V$ denote the singular value
 decomposition, where $U$ and $V$ are orthogonal, and $D$ diagonal
 with diagonal entries ranging between $\sigma_{\mathrm{min}}(j)$ and
 $\sigma_{\mathrm{max}}(j)$. To simplify notation, we suppress the dependence of
 $j$ in the singular values. Denoting by $B_{1/2}(0)$ the open ball around zero with respect to the euclidean norm, we have the inclusions
 \[
  (-1/2,1/2)^n \subset 2 \sqrt{n} B_{1/2}(0) \subset 2n (-1/2,1/2)^n~.
 \]
 This gives the following chain of inclusions
 \begin{eqnarray*}
   \sigma_{\mathrm{min}}  U^{-1}(-1/2,1/2)^n & \subset&  2 \sqrt{n} \sigma_{\mathrm{min}} U^{-1} B_{1/2}(0) \\ & \subset &  2n \sigma_{\mathrm{min}} (-1/2,1/2)^n ~.
 \end{eqnarray*}
On the other hand, we have $\sigma_{\mathrm{min}}(B_{1/2}(0)) \subset D (B_{1/2}(0))$, and thus, since $V B_{1/2}(0) = B_{1/2}(0)$, we get
\begin{eqnarray*}
 \sigma_{\mathrm{min}}(-1/2,1/2)^n & \subset&  2 \sqrt{n} \sigma_{\mathrm{min}} B_{1/2}(0) \subset  2 \sqrt{n} \sigma_{\mathrm{min}}  V B_{1/2}(0) \\ & \subset & 2 \sqrt{n} DV B_{1/2}(0)
 \subset 2n DV (-1/2,1/2)^n ~.
\end{eqnarray*} Combining these inclusions yields
\[
 \frac{\sigma_{\mathrm{min}}}{4n^2} (-1/2,1/2)^n \subset \frac{\sigma_{\mathrm{min}}}{2n} U(-1/2,1/2)^n \subset UDV (-1/2,1/2)^n = C_j^{-T} (-1/2,1/2)^n~.
\]
Furthermore, recalling the dependence of $j$, we have
\[
 |\det(C_j)^{-T}| \le \sigma_{\mathrm{max}}(j)^n \le K^n \sigma_{\mathrm{min}}(j)^n 
\] via (\ref{eqn:aniso}),
and thus the infinite bandwidth assumption yields
\[
 \sum_{j \in J}  \sigma_{\mathrm{min}}(j)^n = \infty~. 
\]
To summarize, we have that the fundamental domains $C_j[-1/2,1/2)^n$ modulo $\Gamma_j^\bot$ contain cubes of  infinite combined volume. 
Now the elementary, but somewhat technical following Lemma \ref{lem:cubes} shows the desired covering property. 
\end{proof}
 
 \begin{lemma} 
 \label{lem:cubes}
  Let $K_j = k_j [0,1)^n$, $j \in \N$, denote a sequence of cubes in $\R^n$, with $\sum_{j \in \N} k_j^n = \infty$. Then there exist vectors $\tau_j \in \R^n$ such that $\R^n = \bigcup_{j \in \N}\tau_j + K_j$. 
 \end{lemma}
 \begin{proof}
 For the following argument, it is helpful to recall the notion of a dyadic cube. By this we mean a subset $2^k (m + [0,1)^n) \subset \R^n$, with $k \in \Z$ and $m \in \Z^n$.  What we need in the following is that each dyadic cube decomposes into disjoint dyadic cubes of smaller size. 

  We first show the simpler statement that there exist $\tau_j$ ($j \in \N)$ such that 
  \[
   [0,1)^n \subset \bigcup_{j \in \N}\tau_j + K_j~.
  \]
  To see this, we first observe that we may assume that $\{ k_j : j \in \N \} \subset \{ 2^m: m \in \Z \}$: If $\tilde{k}_j$ denotes the largest power of $2$ that is less than or equal to $k_j$, then we have 
$\sum_{j} \tilde{k}_j^n = \infty$ as well, and any covering by the smaller cubes solves the problem, as well.  

Secondly, note that the problem is easy to solve if the $k_j$ do not converge to zero. In that case, there is either an unbounded subsequence (in which case a single cube from the system can cover the unit cube), or there exist infinitely many cubes of the same size, which then can be used to cover the unit cube. 

Hence, possibly after reindexing the sequence, we are left with the case where $(k_j)_{j \in \N}$ is a decreasing sequence of powers of $2$, converging to zero. Here we can inductively pick $\tau_j$, $j=1,\ldots, $ with the property that for all $\ell \in \N$ satisfying 
\begin{equation} \label{eqn:measure_left} [0,1)^n \setminus \bigcup_{j\le\ell} \tau_j + k_j [0,1)^n \not= \emptyset \end{equation}
we have 
\[
\bigcupdot_{j=1}^{\ell+1} \tau_j + k_j [0,1)^n \subset [0,1)^n ~.
\]
To see this, we pick $\tau_1 = 0$. Assuming that $\tau_1,\ldots,\tau_\ell$ are determined, and (\ref{eqn:measure_left}) holds for $\ell$, we note that by the inductive assumption, $[0,1)^n \setminus \bigcup_{j\le\ell} \tau_j + k_j [0,1)^n$ is the complement of a union of dyadic cubes in $[0,1)$, with side-lengths greater than or equal to $k_{\ell}$, which in turn is greater than or equal to $k_{\ell+1}$. In particular, if this complement is nonempty, it is the union of dyadic cubes of side-length $k_{\ell+1}$. Hence there exists $\tau_{\ell+1} = 2^{k_{\ell+1}} m$, with $m \in \Z^n$, with the desired property. 

Since the volumes of the cubes add up to infinity, the condition (\ref{eqn:measure_left}) can only hold for finitely many $\ell$. Hence we have $[0,1)^n \subset \bigcup_{j =1}^N \tau_j + k_j [0,1)^n$, for sufficiently large $N$. 

In order to cover all of $\R^n$ by shifted cubes, we reindex the sequence $(k_j)_{j \in \N}$ into a double sequence $(r_{j,\ell})_{(j,\ell) \in \N^2}$ with the property that, for all $j \in \N$,
$\sum_{\ell \in \N} r_{j,\ell}^n = \infty$. Numbering the cubes of the type $m + [0,1)^n$, with $m \in \Z^n$, as $(M_j)_{j \in \N}$, the first step of the proof shows that we can cover $M_j$ using the cubes with side-lengths $(t_{j,\ell})_{\ell \in \N}$. Hence we have achieved the desired covering of $\R^n$ using the full family of cubes.
 \end{proof}

\section{Finite system bandwidth}
\label{sec:system-bandwidth}

By the intuition outlined in the introduction, and substantiated in
Theorem~\ref{thm:density-GTI}, large bandwidth $BW(\mathcal{G})\ge \mu_{\ghat}(\ghat)$
is necessary for the existence of tight frame generators. On
non-discrete groups even \emph{infinite} bandwidth is necessary. Note however that these conclusions required additional assumptions, in the form of LIC or UCP. We will see that without these assumptions, the bandwidth intuition fails. Surprisingly, we will even see that
$BW(\mathcal{G})$ can be arbitrarily small, while still preserving the
frame property; actually, even orthonormal bases can have arbitrarily
small system bandwidth. 

\begin{theorem} \label{thm:ONB_fbw}
Assume that there exists a sequence $(\Gamma_n)_{n \in \mathbb{N}_0}$ of strictly decreasing lattices $\Gamma_0 \supsetneq \Gamma_1 \supsetneq \ldots$ in $G$. Then,
 given any $\epsilon>0$, there exists a system $\mathcal{L}= (\Lambda_i)_{i \in I}$ of lattices in $G$ with $BW(\mathcal{L}) < \epsilon$, and a system of functions $(h_i)_{i \in I}$ such that the associated GSI system is an orthonormal basis. 
\end{theorem}

The remainder of this section will prove the construction. The following results exploit an idea introduced by Bownik and Rzeszotnik in \cite{MR2066821}, namely that it is possible to index the same family of vectors as GSI system over different lattice families. While this relabeling does not affect any pertinent property of the associated frame operator(s), it may influence other properties of the system, most notably its bandwidth.
\begin{definition}
 Let $\mathcal{G} = (\Gamma_j)_{j \in J}$ and $\mathcal{L} = (\Lambda_i)_{i \in I}$ denote lattice systems. We say that $\mathcal{L}$ is a \emph{refinement} of $\mathcal{G}$ if there exists a partition $(I_j)_{j \in J}$ of $I$ and vectors $(\gamma_i)_{i \in I}$ with the property 
 \[
  \forall j \in J~:~\Gamma_j = \bigcupdot_{i \in I_j} \gamma_i + \Lambda_i~.
 \]
\end{definition}

We then have the following obvious fact.
\begin{lemma} \label{lem:reindex}
 Let $\mathcal{G} =  (\Gamma_j)_{j \in J}$ be a system of lattices, and  $\mathcal{L} = (\Lambda_i)_{i \in I}$ a refinement of $\mathcal{G}$. 
 Given a system of functions $(g_j)_{j \in J}$, and define
 \[
  h_{i} = T_{\gamma_i} g_j~,i \in I_j~.
 \] Then the GSI system $(T_\lambda h_{i})_{i \in I,\lambda \in
   \Lambda_i}$ is obtained by reindexing the GSI system $(T_\gamma
 g_j)_{j \in J,\gamma \in \Gamma_j}$. In particular, it is a Bessel
 system, a tight/Parseval frame, or an orthonormal basis if and only if the original system has the same properties. This observation extends to dual systems.
\end{lemma}

\begin{remark}
 Note that whenever one has 
 \[ \Gamma \supset \bigcupdot_{i \in I} \gamma_i + \Lambda_i \]
 with finitely many lattices $\Lambda_i$, $i \in I$, then 
 \[
  \frac{1}{\covol(\Gamma)} \ge \sum_{i \in I} \frac{1}{\covol(\Lambda_i)}~.
 \]
 Hence, if $\mathcal{L}$ is a refinement of $\mathcal{G}$, then one has that 
 \[
 BW(\mathcal{L}) \le BW(\mathcal{G})~.
 \] The whole point of introducing refinements to our discussion is the fact that this inequality can be proper.
 It is also worth noting that whenever the index sets $I_j$ occurring in a refinement are all finite, the bandwidth does not change. 
\end{remark}

We next show that refinements can be constructed from chains of subgroups. Note that Lemma~\ref{lem:grp_refine} is valid also for non-abelian groups, but we only formulate it for the setting we need.  
\begin{lemma} \label{lem:grp_refine}
 Let $H$ denote a countable abelian group, and let $(H_j)_{j \in \mathbb{N}}$ denote a sequence of proper subgroups with finite index, and $H_j \supsetneq H_{j+1}$ for all $j \in \mathbb{N}$. Then there exists a sequence $(\gamma_j)_{j \in \mathbb{N}} \subset H$ such that
 \[
  H = \bigcupdot_{j \in \N} \gamma_j +  H_j~.
 \]
\end{lemma}
\begin{proof}
 Let $(h_k)_{k \in \mathbb{N}}$ denote an enumeration of $H$. We choose the $\gamma_j$ inductively, with $\gamma_1 = h_1$. 
 Then $H \setminus \gamma_1 + H_1$ is nonempty.
 
 Assume that after $j$ steps, we have found $\gamma_1,\ldots, \gamma_j$ such that 
 \[
 K_j = H \setminus \bigcupdot_{ \ell \le j} \gamma_{\ell} + H_{\ell}
 \] 
is nonempty. Since $H_j \subset H_{\ell}$ for all $\ell < j$, $K_j$ is
a union of $H_j$-cosets. Now pick $k \in \mathbb{N}$ minimal with $h_k
\not\in K_j$, and let $\gamma_{j+1} = h_k$. Then $\gamma_{j+1} +
H_{j+1} \subset K_j$, because $H_j \supset H_{j+1}$, and thus
$(\gamma_{j+1} + H_{j+1} \cap \gamma_{\ell} +  H_{\ell}) \subset (K_j \cap \gamma_{\ell} + H_{\ell}) = \emptyset$, for all $\ell \le j$. Finally, the fact that $H_j \supsetneq H_{j+1}$ implies that 
 \[
  K_{j+1} =  H \setminus \bigcupdot_{ \ell \le j+1} \gamma_{\ell} + H_{\ell}
 \] is a nonempty union of $H_{j+1}$-cosets. 
 
 Thus the inductive procedure can be continued to yield a sequence
 $(\gamma_j)_{j \in \mathbb{N}}$ with $\gamma_j + H_j \cap \gamma_\ell + H_\ell = \emptyset$ for $j \not= \ell$. In addition, the choice of $\gamma_{j+1}$ in the induction step allows to prove inductively that $h_k \in \bigcup_{j \le k} \gamma_j + H_j$, for all $k \in \mathbb{N}$. Hence the sequence has all the desired properties. 
\end{proof}

\begin{proof}[Proof of Theorem \ref{thm:ONB_fbw}]
First consider a constant lattice system $\mathcal{G} = (\Gamma_0)_{\alpha \in \Gamma_0^\bot}$.
If $K \subset \ghat$ is any fundamental domain modulo $\Lambda_0^\bot$, then its translates $K_\alpha = \alpha + K$, with $\alpha \in \Lambda_0^\bot$ are a disjoint covering of $\ghat$, and
Theorem \ref{thm:shannon} provides the existence of a family of
orthonormal basis generators for $\mathcal{G}$. 

We will now construct a refinement $\mathcal{L}$ of $\mathcal{G}$ with finite bandwidth, as follows: Let $I = \Gamma_0^\bot \times \mathbb{N}$, and fix a bijection $\varphi: I \to \mathbb{N}$ 
with the property that, for every $\alpha \in \Gamma_0^\bot$, the sequence $\varphi(\alpha,\cdot)$ is strictly increasing. Given $i = (\alpha,k) \in I$, let $\Lambda_i = \Gamma_{\varphi(i)}$. 
By choice of $\varphi$, we have for all $\alpha \in \Gamma_0^\bot$, that 
\[
 \Gamma_0 \supsetneq \Lambda_{\alpha,1} \supsetneq \Lambda_{\alpha,2} \supsetneq \ldots~.
\]  Hence Lemma \ref{lem:grp_refine} implies that $\mathcal{L}$ is a refinement of $\mathcal{G}$, and since $\varphi$ is bijective, the refined system has bandwidth 
\[
 BW(\mathcal{L}) = \sum_{i \in I}\frac{1}{\covol(\Lambda_i)} = \sum_{n \in \mathbb{N}} \frac{1}{\covol(\Gamma_n)}. 
\] Now for every $n \in \mathbb{N}$, the inclusion $\Gamma_n \subset
\Gamma_{n-1}$ yields $\covol(\Gamma_n) = [\Gamma_{n}:\Gamma_{n-1}]
\covol(\Gamma_{n-1})$. Since the inclusions are strict, all subgroup
indices are at least $2$, which leads to $\covol(\Gamma_n) \ge 2^n \covol(\Gamma_0)$, and we finally arrive at
\[
 BW(\mathcal{L}) \le \sum_{n \in \mathbb{N}} \frac{1}{2^n \covol(\Gamma_0)} = \frac{1}{\covol(\Gamma_0)}
\]
Hence, there exist GSI orthonormal bases with bandwidth
$1/\covol(\Gamma_0)$. 

Now, starting from the constant lattice system $\mathcal{G} =
(\Gamma_m)_{\alpha \in \Gamma_m^\bot}$ in the above construction, we
obtain orthonormal bases with bandwidth less than or equal to $1/\covol(\Gamma_m)$. 
\end{proof}
Lemma~\ref{lem:reindex} showed that a system of frame, Bessel, or dual
frame generators for a system $\mathcal{G}$ can be used to provide a
system of generators for $\mathcal{L}$, whenever $\mathcal{L}$ is a
refinement of $\mathcal{G}$. The converse is generally not true as the
following example shows. 
\begin{example} \label{exa:reindex}
The system $\mathcal{L}  = (2^j \mathbb{Z})_{j \in \mathbb{N}}$ is a
refinement of the single lattice system $\mathcal{G} =
(\mathbb{Z})_{i=1}$. By Theorem~\ref{thm:ONB_fbw}, there exist
orthonormal basis
generators in $L^2(\mathbb{R})$ for $\mathcal{L}$, but by
Corollary~\ref{cor:finite_ls}, $\mathcal{G}$ has no frame generators. 
\end{example}

\section{Independent lattices and UCP}
\label{sec:indep-latt-ucp}

The aim of this section is to exhibit a general setup for which
$\infty$-UCP holds as soon as $w_{f;g,h}$ is bounded. 
Furthermore, we argue that this setup is the generic case of GSI
systems and that it leads to rather stringent condition on the frame generators. 

\begin{definition}
A lattice system $(\Gamma_j)_{j \in J}$ is called \emph{independent}
if for all families $(x_j)_{j \in J'}$ with finite $J' \subset J$ and $x_j \in \Gamma_j$, we have the implication
\[
\sum_{j \in J'} x_j = 0 \Rightarrow \forall j \in J':\,\, x_j = 0\,.
\]
We call the system \emph{pairwise independent} if $\Gamma_j \cap \Gamma_k = \{ 0 \}$, whenever $j \not= k$. 
\end{definition}
We will be interested in lattice families whose dual lattices are independent. The following lemma characterizes this condition in terms of a density property.
\begin{lemma} \label{lem:independent_dense}
Let $(\Gamma_j)_{j \in J}$ denote a family of lattices in $G$. Then the following are equivalent:
\begin{enumerate}[(i)]
\item The dual lattices $(\Gamma_j^\bot)_{j \in J}$ are independent.
\item For all finite subsets $J'$, the subgroup
\[
\{ (x + \Gamma_j)_{j \in J'} : x \in G \} \subset \prod_{j \in J'} G/\Gamma_j
\] is dense with respect to the product topology.
\item The subgroup
\[
\{ (x + \Gamma_j)_{j \in J} : x \in G \} \subset \prod_{j \in J} G/\Gamma_j
\] is dense with respect to the product topology. 
\end{enumerate}
\end{lemma}

\begin{proof}
For the proof of $(i) \Leftrightarrow (ii)$, consider the continuous group homomorphism $\varphi : G \to \prod_{j \in J'} G/\Gamma_j$, defined by $\varphi(x) = (x + \Gamma_j)_{j \in J'} $. Let 
\[
\hat{\varphi} : \left( \prod_{j \in J'} G/\Gamma_j \right)^\wedge \to \ghat
\]
denote the dual homomorphism, defined by
\[
\langle x, \hat{\varphi}(\alpha) \rangle = \langle \varphi(x), \alpha \rangle\,\,.
\]
We may identify $ \left( \prod_{j \in J'} G/\Gamma_j \right)^\wedge$ with $\prod_{j \in J'} \Gamma_j^\bot$, using the duality 
\[
\langle (x_j + \Gamma_j)_{j \in J'}, (\alpha_j)_{j \in J'} \rangle = \prod_{j \in J'} \langle x_j, \alpha_j \rangle\,\,.
\]
With this identification, we obtain
\begin{eqnarray*}
\langle x, \hat{\varphi}((\alpha_j)_{j \in J'}) \rangle & = & \langle (x + \Gamma_j)_{j \in J'}, (\alpha_j)_{j \in J'} \rangle \\
& = & \prod_{j \in J'} \langle x, \alpha_j \rangle  \\ 
& = & \langle x, \sum_{j \in J'} \alpha_j \rangle\,\,,
\end{eqnarray*} leading to 
\[
\hat{\varphi}((\alpha_j)_{j \in J'}) = \sum_{j \in J'} \alpha_j \,\,.
\] 
Now (i) is equivalent to injectivity of $\hat{\varphi}$, for all choices of finite $J' \subset J$, whereas (ii) is equivalent to the fact that $\varphi$ has dense image. But the statements about the homomorphisms are equivalent by duality theory: If $\varphi$ has dense image, then two continuous functions (for example, characters) coinciding on $\varphi(G)$ must coincide everywhere, which shows (ii) $\Rightarrow$ (i). And if the image of $\varphi$ is not dense, there exists a nontrivial character on the quotient $ \left( \prod_{j \in J'} G/\Gamma_j \right)/\overline{\varphi(G)}$, which gives rise to a character on $ \left( \prod_{j \in J'} G/\Gamma_j \right)$ that coincides on $\varphi(G)$ with the trivial character, showing that $\hat{\varphi}$ is not injective, and thus (ii) $\Rightarrow$ (i). 

Finally, the equivalence (ii) $\Leftrightarrow$ (iii) is a standard fact about product topologies. 
\end{proof}

\begin{remark}
For $G = \mathbb{R}$ and $\Gamma_i = c_i \mathbb{Z} \subset G$,
independence of the dual lattices is equivalent to linear independence
of $(1/c_j)_{j \in J}$ over the rationals. We note that this is not
the same as linear independence of $(c_j)_{j \in J}$ over the
rationals. E.g., if $c \in \mathbb{R}$ is transcendental, then the
family $(c_n)_{n \in \N}$ given by $c_n = n+c$, $n \in \mathbb{N}$, is linearly dependent over the rationals, but $(1/c_n)_{n \in \mathbb{N}}$ is not. 

While the condition of rational independence may seem strong, one can argue that in a sense,
it is the generic case: If one chooses the lattice generators $c_j$ randomly, with independent Lebesgue-absolute continuous
probability densities for each $j \in J$, then the system $(1/c_j)_{j \in J}$ will be rationally independent with probability one.

If $G = \mathbb{Z}$ and $\Gamma_j =  c_j \mathbb{Z} \subset G$, then the dual lattices are independent if and only if the $c_j$ are pairwise prime. This can be seen by Lemma~\ref{lem:independent_dense} and  the  Chinese Remainder Theorem, stating that
\[
\mathbb{Z} \ni n \mapsto (n + c_j \mathbb{Z})_{j \in J'} \in \prod_{j} \mathbb{Z}/c_j \mathbb{Z}
\] is onto if and only if the $c_j$ are pairwise prime. 
\end{remark}

The following theorem shows the scope of the $\infty$-UCP condition. 
 \begin{theorem} \label{thm:rat_ind}
  Let $(\Gamma_j)_{j \in J}$ denote a system of lattices, and let
  $(g_j)_{j \in J},(h_j)_{j \in J} \subset L^2(G)$.  
  \begin{enumerate}[(i)]
   \item Suppose that the dual lattices are independent. Let $f \in
     \cD$. Then $w_{f;g}
     \in L^\infty(G)$  if and only if $w_{f;g}=\sum_{j \in J}
     w_{f;g,j}$ converges uniformly. 
     In particular, every Bessel family $(g_j)_{j \in J}$
     fulfills the $\infty$-UCP with respect to any closed set $E \subset \ghat$ of measure zero. \label{item:ind-lat-UCP}
    \item Suppose that the dual lattices are pairwise
      independent. Then two families $(g_j)_{j \in J}$ and $(h_j)_{j \in J}$ of Bessel generators satisfying the $1$-UCP are dual frame generators if and only if 
 \[
  \sum_{j \in J}  \frac{1}{\covol(\Gamma_j)} \overline{\hat{g}_j(\omega)} \hat{h}_j(\omega) = 1~,  
 \]
and
 \[
  \forall j \in J \forall \alpha \in \Gamma_j^\perp \setminus \{ 0 \}~:~ \overline{\hat{g}_j(\omega)} \hat{h}_{j}(\omega + \alpha) = 0~,
 \] for almost all $\omega \in \ghat$. \label{item:ind-lat-dual}
  \end{enumerate}
 \end{theorem}

 \begin{proof}
Fix $f \in \cD$. If $\infty$-UCP holds, then $w_{f;g} \in \AP(G) \subset
L^\infty(G)$. To finish the proof of \eqref{item:ind-lat-UCP}, we need to show that
  \begin{equation} \label{eqn:wf1sum}
 w_{f;g} = \sum_{j \in J} w_{f;g,j} 
\end{equation} converges uniformly. 
In fact, we will show that 
\begin{equation} \label{eqn:abs_conv}
 \sum_{j \in J} \| w_{f;g,j} \|_\infty < \infty~.
\end{equation} 

For this purpose,  fix $0< \epsilon < 1$ and a finite, nonempty set $J' \subset J$. 
Each $w_{f;g,j}$ induces a continuous function on the compact group $G/\Gamma_j$, hence there exists an open set $M_j \subset G$ such that 
\[
\forall j \in J'~,\forall y \in M_j + \Gamma_j~:~w_{f;g,j}(y) \ge (1-\epsilon) \| w_{f;g,j} \|_\infty~. 
\]
By the independence assumption on the dual lattices and
Lemma~\ref{lem:independent_dense}, there exists $x \in \cap_{j \in J'}
M_j + \Gamma_j$. Since all $w_{f;g,j}$ are positive, and their sum is
pointwise bounded by the $\norm[\infty]{w_{f;g}}$, we get 
\[
 \norm[\infty]{w_{f;g}} \ge \sum_{j \in J'} w_{f;g,j}(x) \ge \sum_{j \in J'} (1-\epsilon)  \| w_{f;g,j} \|_\infty~.
\]
Since $0 < \epsilon < 1$ and $J' \subset J$ were chosen arbitrary, (\ref{eqn:abs_conv}) is shown, and thus part \eqref{item:ind-lat-UCP}.

For the remainder of the proof, it is enough to observe that the characterizing equations from Theorem~\ref{thm:meta_convergence_AP}, which are applicable by the $1$-UCP assumption, simplify to the form given in \eqref{item:ind-lat-dual}, when the dual lattices are pairwise independent. 
\end{proof}

The point of the following result is that if a family of lattices has pairwise independent dual lattices, and there exist dual frame generators  $(g_j)_{j \in J},(h_j)_{j \in J}$ , then the somewhat simple-minded procedure from Theorem \ref{thm:shannon} will also provide such generators. 
\begin{corollary} \label{cor:pw_independent_shannon}
Assume that $(\Gamma_j)_{j \in J}$ is a family of lattices with pairwise independent dual lattices. If there exist dual frame generators  $(g_j)_{j \in J},(h_j)_{j \in J}$ satisfying the $1$-UCP, then there exist Borel sets
$(K_j)_{j \in J}$ satisfying 
 \begin{enumerate}[(i)]
  \item $\mu_{\ghat}(K_j \cap \gamma + K_j) =0$, for all $\gamma
    \in \Gamma_j^\bot \setminus \{ 0 \}$ and for all $j \in J$, 
  \item $\mu_{\ghat}(\ghat \setminus \bigcup_{j \in J}
    K_j) = 0$.  
 \end{enumerate}
\end{corollary}
\begin{proof}
 This follows from the characterizing equations in Theorem \ref{thm:rat_ind}\eqref{item:ind-lat-dual}, if we let 
 \[
  K_j = \{ \omega \in \ghat : \hat{g}_j(\omega) \hat{h}_j(\omega) \not= 0\}
 \] 
for each $j \in J$
\end{proof}

 If one restricts further to orthonormal basis generators, the characterizing equations become even more stringent. 
 \begin{corollary}
\label{cor:ind-lat-char-fun}
 Let $\mathcal{G} = (\Gamma_j)_{j \in J}$ denote a system of lattices whose dual lattices are pairwise independent.  
 Let $(g_j)_{j \in J}$ an associated system of orthonormal basis generators fulfilling $1$-UCP. Then
  \[
  | \hat{g}_j| = c_{j}^{1/2} \mathbf{1}_{K_j} 
  \] where $K_j$ is a measurable fundamental domain mod $\Gamma_j^\bot$, and, up to sets of measure zero,  
  \[
   \ghat = \bigcupdot_{j \in J} K_j~.
  \]  
 \end{corollary}
\begin{proof}
 Let $K_j = \hat{g}_j^{-1}(\mathbb{C}\setminus \{ 0 \})$. Then
 Theorem~\ref{thm:rat_ind}\eqref{item:ind-lat-dual} implies that, up
 to a set of measure zero, the set $K_j$ is contained in a fundamental domain modulo $\Gamma_j^\bot$.  Now the fact that the $\Gamma_j$-shifts of $g_j$ are an orthonormal system forces $K_j$ to have measure $1/\covol(\Gamma_j)$, and that $ | \hat{g}_j| = c_{j}^{1/2} \mathbf{1}_{K_j} $, with $c_j = \covol(\Gamma_j)$. Thus the $\Gamma_j$-shifts of $g_j$ are an orthonormal basis of 
 \[
  \mathcal{H}_j = \{ f \in L^2(G) :  \hat{f} \cdot \mathbf{1}_{K_j} = \hat{f} \}~. 
 \]
 The assumption that the full system $(T_\gamma g_j)_{j,\gamma}$ is
 orthonormal therefore forces the $\mathcal{H}_{j}$ to be pairwise orthogonal, and thus the $K_j$ to be essentially disjoint. Finally, it is clear that completeness of the system forces $\ghat = \bigcup_{j \in J} K_j$ up to sets of measure zero.
\end{proof}

As a further application of Theorem \ref{thm:rat_ind}, we now construct an example
of a lattice family in dimension two with infinite bandwidth, but without dual frame generators. 

\begin{example} \label{ex:counter_inf_bw}
 Fix a transcendental number $c>1$, and let $\Gamma_j = C_j \mathbb{Z}^2$, where 
\[ 
C_j =
\begin{pmatrix}
c^{-j} & 0 \\
 0 & c^{j}  
\end{pmatrix}
 \quad  
\text{for $j \in \mathbb{N}$.} 
 \]
Hence $\sum_{j \in \mathbb{N}} \frac{1}{\covol(\Gamma_j)} = \infty$, but there do not exist families of dual generators for this system. 
To see this, assume otherwise. Note that by choice of $c$, the dual
lattices are independent, hence Theorem
\ref{thm:rat_ind}\eqref{item:ind-lat-UCP} implies that the dual
generators fulfill $\infty$-UCP. Hence Corollary
\ref{cor:pw_independent_shannon} applies and yields Borel sets
$(K_j)_{j \in \mathbb{N}}$ with $\lambda(\mathbb{R}^2 \setminus
\bigcup_{j \in \mathbb{N}} K_j) = 0$ and $\lambda(K_j \cap (\gamma +
K_j)) = 0$ for all $j \in \mathbb{N}$ and all $\gamma \in
\Gamma_j^\bot = c^j \mathbb{Z} \times c^{-j} \mathbb{Z}$, where
$\lambda$ denotes the Lebesgue measure on $\R^2$. 

Without loss of generality, we may assume that $K_j \cap \gamma + K_j = \emptyset$, for all $j \in \mathbb{N}$ and $\gamma \in \Gamma_j^\bot$. Define, for $j \in \mathbb{N}$ and $x \in \mathbb{R}$, the Borel set
 \[
  G_{j,x} = \{ y \in \mathbb{R}: (x,y) \in K_j \}~.
 \] Assume that there exists $k \in \mathbb{Z} \setminus \{ 0 \}$ such that $G_{j,x} \cap (G_{j,x} +c^{-j} k) \not= \emptyset$. Then there exists $y \in \mathbb{R}^n$ such that $(x,y) \in K_j$ and $(x,y-c^{-j}k) \in K_j$. Hence $(x,y) \in K_j \cap ((0,c^{-j}k) + K_j)$, and $(0,c^{-j}k) \in \Gamma_j^\bot$, which contradicts our assumption on the $K_j$.
 
 Hence $G_{j,x}$ is contained in a fundamental domain mod $c^{-j} \mathbb{Z}$, which entails $\lambda(G_{j.x})  \le c^{-j}$. 
 Now, let us assume that $\mathbb{R}^2 \subset \bigcup_{j \in \mathbb{N}} K_j$, up to a null set. We then get
 \begin{eqnarray*}
  \infty  & = &  \int_0^1 \int_{\mathbb{R}} 1 dydx \\
  & = &  \int_0^1 \int_{\mathbb{R}}  \sum_{j \in \mathbb{N}} \mathbf{1}_{K_j} (x,y) dy dx \\
  & = & \sum_{j=1}^\infty \int_0^1 \underbrace{\lambda(G_{j,x})}_{\le c^{-j}} dx \\
  & \le & \sum_{j=1}^\infty c^{-j} < \infty ~,
 \end{eqnarray*}
which is the desired contradiction. 
\end{example}

\begin{remark}
 The results in this section align nicely with results from wavelet analysis. For example, Corollary~\ref{cor:ind-lat-char-fun} is related to so-called MSF wavelets. Such wavelets $\psi$ are characterized by the property that $|\hat{\psi}|$ is, up to scalar multiplication, given by the characteristic function of a Borel set. It was shown by Chui and Shi in \cite{MR1793418}, that whenever the dilation $a>1$ is such that all integer powers of $a$ are irrational, every orthonormal wavelet associated to $a$ must be an MSF wavelet. Corollary~\ref{cor:ind-lat-char-fun}, applied to the family $\Gamma_j = a^j \mathbb{Z}$, for $j \in \mathbb{Z}$, provides this answer under the strictly stronger assumption that $a$ is transcendental (which is equivalent to independence of the dual lattices). Note however that here our corollary also provides a stronger conclusion, since the generators $(g_j)_{j \in \mathbb{Z}}$ are not assumed to be dilates of a single mother wavelet. 
 
 But also Theorem~\ref{thm:rat_ind}\eqref{item:ind-lat-UCP} and its proof have a precedent in wavelet analysis. Note that the proof of the Theorem yields 
 \begin{equation} \label{eqn:sum_bds}
  \| w_{f;g,h} \|_{\infty} = \sum_{j \in J} \| w_{f;g,h,j} \|_{\infty}~.
 \end{equation} This phenomenon is related to the question how to
 estimate frame bounds of the full system $(T_\gamma g_j)_{\gamma \in
   \Gamma_j,j \in J}$ from the bounds of the individual layers 
 $(T_\gamma g_j)_{\gamma \in \Gamma_j}$ indexed by $j \in J$, which was investigated for wavelet
 systems with transcendental dilations in \cite{MR2280189}. Indeed,
 \gsi\ is a Bessel system with optimal bound $B^\dagger$ precisely
 when 
\[ 
  B^\dagger = \sup_{f \in \cD, \norm{f}=1} \sum_{j \in J}
  \max_{x\in G}{w_{f;g,j}(x)}<\infty;
\]
which furthermore is a frame with optimal lower bound $A^\dagger$ if
and only if
\[ 
  A^\dagger = \inf_{f \in \cD, \norm{f}=1} \sum_{j \in J} \min_{x\in G}{w_{f;g,j}(x)}>0.
\]
 These estimates should be compared to \eqref{eq:wf-opt-bounds-B} and \eqref{eq:wf-opt-bounds-A}; 
 they can be viewed as a generalization of  \cite[Theorem
 2.1]{MR2280189}.
\end{remark}

First an example concerning
perturbation (in)stability of this property. Indeed, the existence of
normalized tight frame generators is \emph{not} robust with respect to
arbitrarily small perturbations of the lattice generators.

\begin{example} \label{exa:robust}
Consider $G =\R$ with the Lebesgue measure.
 Consider the system $\mathcal{G} = (2^j \Z)_{j \in \N}$, and let $(\epsilon_j)_{j \in \N}$ be an arbitrary sequence of strictly positive numbers. Pick a sequence $(c_j)_{j \in \mathbb{N}} \subset \R$ with $|2^j - c_j|<\min(\epsilon_j,1)$, and the additional property that $(1/c_j)_{j \in J}$ is $\Q$-linearly independent (this is easily done inductively). Then Theorem~\ref{thm:ONB_fbw} yields a system of tight frame (even orthonormal basis) generators associated with $\mathcal{G}$. However, for the perturbed lattice system $\mathcal{G}' = (c_j \Z)_{j \in \N}$, we can estimate $BW(\mathcal{G}') \le 2$, hence Theorem~\ref{thm:rat_ind} shows that no generators can exist for $\mathcal{G}'$. 
\end{example}

A question that is somewhat similar to the notion of refinements of lattice families is whether the existence of frame generators is robust with respect to enlarging each lattice in the family individually. At first glance, this may seem like a reasonable conjecture; after all, enlarging the lattices leads to systems with more redundancy (and larger bandwidth), which should make frame construction easier. However, this intuition is generally misleading, as the following example shows.

\begin{example} \label{exa:increase}
Consider $G = \Z$ with the counting measure. Fix a family $(c_j)_{j
  \in \mathbb{N}}$ of pairwise prime integers such that $\sum_{j \in
  \mathbb{N}} c_j^{-1} <1 = \mu_{\ghat}(\ghat)$. Then the tight
lattices are independent. Hence, there does not exist a family of dual
frame generators in $\ell^2(\Z)$ for the lattices $\Gamma_j = c_j \Z$
by Theorem~\ref{thm:rat_ind}\eqref{item:ind-lat-UCP} and Theorem~\ref{thm:density-GTI}. On the other hand, if we let
\[
\Lambda_j = \bigcap_{i \le j} \Gamma_i
\] we obtain a strictly decreasing family of lattices. By the proof of Theorem \ref{thm:ONB_fbw}, there is a system of dual frame generators for the $\Lambda_j$, and $\Lambda_j \subset \Gamma_j$ holds for all $j \in \mathbb{N}$.  Thus increasing the lattices can have a \emph{negative} impact on the availability of tight frame generators.
\end{example}

\section{Applications and extensions }
\label{sec:appl-extens-}

We end this paper with further discussions of the necessary
conditions for the frame property in Section~\ref{sec:necess-cond}. 

Intuitively, the Calder\'on sum $ \sum_{j \in J}\tfrac{1}{\covol(\Gamma_{\!j})}
     \abs{\hat{g}_j(\cdot)}^2 $ measures the total energy
concentration of the generators $g_j$ in the frequency domain. If the
Calder\'on sum is zero on some domain in
frequency, then clearly none of the frequencies in this domain can be represented
by the corresponding GSI system. In other words, the corresponding GSI
system is not complete/total. Furthermore, whenever a GSI system
has the frame property, which is a stronger assumption than the spanning property,
one would even expect the Calder\'on sum to be bounded uniformly from below since
the GSI frame can reproduce all frequencies in a \emph{stable} way.

However, as we saw in Theorem~\ref{thm:a-lic-calderon-bounds} and
Example~\ref{ex:UCP_LIC}, this is again a situation where our
intuition only holds true if we assume the $1$-UCP. Under the $1$-UCP, the
Calder\'on sum of a GSI frame
$\gsi$  with bounds $A_g$ and $B_g$ takes values in
$\itvcc{A_g}{B_g}$, that is,
\[
     A_g \le \sum_{j\in J}  \frac{1}{\covol(\Gamma_{\!j})}
     \abs{\hat{g}_j(\omega)}^2  \le B_g \quad \text{for a.e. }  \omega \in \ghat.
\]  
Without the $1$-UCP, the best one can say is that 
\[
     0 < \sum_{j\in J}  \frac{1}{\covol(\Gamma_{\!j})}
     \abs{\hat{g}_j(\omega)}^2  \le B_g \quad \text{for a.e. }  \omega \in \ghat.
\]

As mentioned, the terminology ``Calder\'on sum'' comes from wavelet
analysis. Let us show that our results on GSI systems extends known
results in wavelet analysis. Fix an $n \times n$ matrix $A \in
\mathrm{GL}_n(\R)$ and a full-rank lattice $\Gamma \subset \R^n$. The wavelet system $\seq{D_{A^j}T_{\gamma}
  \psi }_{j \in \Z, \gamma \in \Gamma}$, where  $D_{A^j}T_{\gamma}
  \psi=\abs{\det{A}}^{j/2} \psi(A\cdot -\gamma)$, can be written as a
  GSI system in the following \emph{standard form}:
 \[ 
 J=\Z, \quad \Gamma_j = A^{-j} \Gamma, \quad g_j=D_{A_j} \psi , \quad
 \text{for all } j \in \Z.
\]
The Calder\'on sum then reads $\sum_{j \in \Z} \abssmall{\hat{\psi}(A^j \cdot)}^2$.
It is a classical result by Chui and Shi~\cite{MR1199539} that for
univariate frame wavelets
 ($n=1$, $A=a$) with bounds $C_1$ and $C_2$, it holds
\[
     C_1 \le \sum_{j \in \Z} \abssmall{\hat{\psi(a^j \omega)}}^2 \le C_2 \quad \text{for a.e. }  \omega \in \R.
\]  
In wavelet analysis the case $n=1$ is special: it is the only
dimension where the LIC/UCP automatically holds once we assume local
integrability in $\R^n \setminus \{0\}$ of the Calder\'on sum. Hence, for univariate wavelets the issue of LIC/UCP
is, in most cases, completely
absent. 

\begin{theorem}
  Let $A \in \mathrm{GL}_n(\R)$, $\abs{\det{A}}>1$, let
  $\LG \subset \R^n$ be a full-rank lattice, and let $L$ be an at most
  countable index set. 
 Suppose that $(A^T,\Gamma^\perp)$ satisfies the lattice counting estimate,
 that is,
\[    \#\abs{ \LG^\perp \cap (A^T)^j(B(0,r)) } \le C \max (1, \abs{\det{A}}^j)
\qquad\text{for all }j\in \Z.
\]
 If the wavelet system $\seq{D_{A^j}T_{\gamma}
  \psi_\ell }_{\ell \in L, j \in \Z, \gamma \in \Gamma}$ is a frame
with bounds $C_1$ and $C_2$, then
\[
     C_1 \le \sum_{\ell \in L}\sum_{j \in \Z} \abssmall{\hat{\psi}_\ell(A^j \omega)}^2 \le C_2 \quad \text{for a.e. }  \omega \in \R^n.
\]  
\end{theorem}
\begin{proof}
  We consider the wavelet system as a GSI system in the standard form.
  By Lemma~\ref{lem:calderon_upper}, it holds that $\sum_{\ell \in L}\sum_{j \in \Z}
  \abssmall{\hat{\psi}_\ell(A^j \omega)}^2 \le C_2$ for a.e. $\omega \in \R^n$.
  Since $(A^T,\Gamma^\perp)$ satisfies the lattice counting estimate,
  this implies, by a result in \cite{BownikNonexpanding2015}, that the wavelet system satisfies that LIC. Since the
  LIC implies $1$-UCP, the result follows from Theorem~\ref{thm:a-lic-calderon-bounds}.
\end{proof}

The lattice counting estimate was introduced in
\cite{BownikNonexpanding2015}, where Bownik and the second named
author show that almost all wavelet
systems satisfy the lattice counting estimate. In particular, 
a dilation matrix $A$ that is \emph{expanding on a subspace} (i.e., matrices with eigenvalues bigger than
one in modulus, at least one strictly bigger, and eigenvalues of modulus one have Jordan blocks of order one) and any translation lattice $\Gamma\subset \R^n$ will satisfy the lattice counting estimate.
 
The proofs of the lower bound of wavelet frames and GSI frames for
$L^2(\R)$ in~\cite{MR1199539} and \cite{MR3627474}, respectively, are of similar nature, and they rely on the
fact that lattices $(c_j \Z)_{j \in \Z}$ in $\R$ has a natural
ordering. Indeed, one can
assume $c_j\le c_{j+1}$. This
is not the case in higher dimensions nor for general LCA groups, and
the mentioned proofs break down. Moreover, the proof of
Theorem~\ref{thm:a-lic-calderon-bounds} is conceptually much simpler
than the proofs in \cite{MR1199539,MR3627474} once the theory of almost
periodic functions of GSI systems is in place. In fact, our proof
extends to a larger class of systems, called \emph{generalized translation-invariant} (GTI) systems, introduced in
\cite{JakobsenReproducing2014}. GTI systems are continuous or
semi-continuous variants of GSI systems, and therefore encompass,
e.g., the continuous (and semi-continuous) wavelet, Gabor and shearlet transforms.
\begin{theorem}
\label{thm:calderon-GTI}
  Let $J$ be at most countable. $(T_\gamma g_{j,p})_{j \in J, p \in
    P_j,\gamma\in \Gamma_j}$ be a GTI system,
  where for each $j \in J$: $g_j \in L^2(G)$, $\Gamma_j \subset G$ is a
  co-compact subgroup (with some given Haar measure), and $P_j$ a
  measure space (satisfying the three standing
  assumptions in \cite{JakobsenReproducing2014}). Suppose $(T_\gamma
  g_{j,p})_{j \in J, p \in P_j,\gamma\in \Gamma_j}$ is a (continuous) frame with bounds
  $A_g$ and $B_g$ that satisfies $1$-UCP (with the
  straightforward modifications). Then
  \[
     A_g \le \sum_{j\in J}  \frac{1}{\covol(\Gamma_{\!j})}
     \abs{\hat{g}_j(\omega)}^2  \le B_g \quad \text{for a.e. }  \omega \in \ghat.
\]  
where $\covol(\Gamma_{\!j}):=\mu_{G / \Gamma_j} (G / \Gamma_j)$ for
each $j \in J$.
    \end{theorem}

Theorems~\ref{thm:meta_convergence_AP},
\ref{thm:meta_convergence_AP_tight} and \ref{thm:density-GTI} also extend to GTI frames as
Theorem~\ref{thm:calderon-GTI}, albeit Theorem~\ref{thm:density-GTI}
needs the additional assumption that $\int_{P_j}\norm{g_{j,p}}^2
d\mu_{P_j}(p) < C$ for all $j \in J$. We remark that there exist LCA
groups that have no lattices, while any LCA group has a co-compact subgroup. 
We leave the existence question of GTI frames for a family of
co-compact subgroups $(\Gamma_j)_{j \in J}$ for future research. 




 \section*{Acknowledgements}
 We thank Mads Jakobsen and Felix Voigtlaender for interesting
 discussions and help with Example 3. We also thank Jordy van
 Velthoven for reading the manuscript and pointing out some typos.


\begin{thebibliography}{10}

\bibitem{MR2854065}
P.~Balazs, M.~D{\"o}rfler, F.~Jaillet, N.~Holighaus, and G.~Velasco.
\newblock Theory, implementation and applications of nonstationary {G}abor
  frames.
\newblock {\em J. Comput. Appl. Math.}, 236(6):1481--1496, 2011.

\bibitem{BownikNonexpanding2015}
M.~Bownik and J.~Lemvig.
\newblock Wavelets for non-expanding dilations and the lattice counting
  estimate.
\newblock {\em Int. Math. Res. Not.}, 2016.

\bibitem{MR2066821}
M.~Bownik and Z.~Rzeszotnik.
\newblock The spectral function of shift-invariant spaces on general lattices.
\newblock In {\em Wavelets, frames and operator theory}, volume 345 of {\em
  Contemp. Math.}, pages 49--59. Amer. Math. Soc., Providence, RI, 2004.

\bibitem{MR3627474}
O.~Christensen, M.~Hasannasab, and J.~Lemvig.
\newblock Explicit constructions and properties of generalized shift-invariant
  systems in {$L^2(\Bbb R)$}.
\newblock {\em Adv. Comput. Math.}, 43(2):443--472, 2017.

\bibitem{MR1793418}
C.~K. Chui and X.~Shi.
\newblock Orthonormal wavelets and tight frames with arbitrary real dilations.
\newblock {\em Appl. Comput. Harmon. Anal.}, 9(3):243--264, 2000.

\bibitem{MR1199539}
C.~K. Chui and X.~L. Shi.
\newblock Inequalities of {L}ittlewood-{P}aley type for frames and wavelets.
\newblock {\em SIAM J. Math. Anal.}, 24(1):263--277, 1993.

\bibitem{MR1066587}
I.~Daubechies.
\newblock The wavelet transform, time-frequency localization and signal
  analysis.
\newblock {\em IEEE Trans. Inform. Theory}, 36(5):961--1005, 1990.

\bibitem{MR1916862}
E.~Hern{\'a}ndez, D.~Labate, and G.~Weiss.
\newblock A unified characterization of reproducing systems generated by a
  finite family. {II}.
\newblock {\em J. Geom. Anal.}, 12(4):615--662, 2002.

\bibitem{HeRo}
E.~Hewitt and K.~A. Ross.
\newblock {\em Abstract harmonic analysis. {V}ol. {I}}, volume 115 of {\em
  Grundlehren der Mathematischen Wissenschaften [Fundamental Principles of
  Mathematical Sciences]}.
\newblock Springer-Verlag, Berlin-New York, second edition, 1979.
\newblock Structure of topological groups, integration theory, group
  representations.

\bibitem{JakobsenReproducing2014}
M.~S. Jakobsen and J.~Lemvig.
\newblock Reproducing formulas for generalized translation invariant systems on
  locally compact abelian groups.
\newblock {\em Trans. Amer. Math. Soc.}, 368(12):8447--8480, 2016.

\bibitem{Kluvanek}
I.~Kluv{\'a}nek.
\newblock Sampling theorem in abstract harmonic analysis.
\newblock {\em Mat.-Fyz. \v Casopis Sloven. Akad. Vied}, 15:43--48, 1965.

\bibitem{MR2283810}
G.~Kutyniok and D.~Labate.
\newblock The theory of reproducing systems on locally compact abelian groups.
\newblock {\em Colloq. Math.}, 106(2):197--220, 2006.

\bibitem{MR1866351}
R.~S. Laugesen.
\newblock Completeness of orthonormal wavelet systems for arbitrary real
  dilations.
\newblock {\em Appl. Comput. Harmon. Anal.}, 11(3):455--473, 2001.

\bibitem{MR1940326}
R.~S. Laugesen.
\newblock Translational averaging for completeness, characterization and
  oversampling of wavelets.
\newblock {\em Collect. Math.}, 53(3):211--249, 2002.

\bibitem{MR2280189}
R.~S. Laugesen.
\newblock On affine frames with transcendental dilations.
\newblock {\em Proc. Amer. Math. Soc.}, 135(1):211--216, 2007.

\bibitem{1704.06510}
J.~Lemvig and J.~T. van Velthoven.
\newblock A sufficient condition and estimates of the frame bounds for
  generalized translation-invariant frames, preprint, arXiv:1704.06510, 2017.

\bibitem{1606.08647}
E.~S. Ottosen and M.~Nielsen.
\newblock A characterization of sparse nonstationary gabor expansions,
  preprint, arXiv:1606.08647, 2016.

\bibitem{1704.07176}
E.~S. Ottosen and M.~Nielsen.
\newblock Nonlinear approximation with nonstationary gabor frames, preprint,
  arXiv:1704.07176, 2017.

\bibitem{Pankov}
A.~A. Pankov.
\newblock {\em Bounded and almost periodic solutions of nonlinear operator
  differential equations}, volume~55 of {\em Mathematics and its Applications
  (Soviet Series)}.
\newblock Kluwer Academic Publishers Group, Dordrecht, 1990.
\newblock Translated from the Russian by V. S. Zaja{\v{c}}kovski [V. S.
  Zayachkovski{\u\i}] and the author.

\bibitem{MR2132766}
A.~Ron and Z.~Shen.
\newblock Generalized shift-invariant systems.
\newblock {\em Constr. Approx.}, 22(1):1--45, 2005.

\bibitem{Rudin_Fourier}
W.~Rudin.
\newblock {\em Fourier analysis on groups}.
\newblock Wiley Classics Library. John Wiley \& Sons, Inc., New York, 1990.
\newblock Reprint of the 1962 original, A Wiley-Interscience Publication.

\end{thebibliography}

 \end{document}